\documentclass[reqno,11pt]{amsart}
\usepackage{amsmath,amssymb,latexsym,soul,cite,mathrsfs,accents}
\usepackage{color,enumitem,graphicx}
\usepackage[colorlinks=true,urlcolor=blue,
citecolor=red,linkcolor=blue,linktocpage,pdfpagelabels,
bookmarksnumbered,bookmarksopen]{hyperref}
\usepackage[english]{babel}
\usepackage[left=2.6cm,right=2.6cm,top=2.9cm,bottom=2.9cm]{geometry}
\usepackage{tensor}
\usepackage{tikz}
\newtheorem{theorem}{Theorem}
\newtheorem{lemma}[theorem]{Lemma}

\newtheorem{proposition}[theorem]{Proposition}
\newtheorem{remark}[theorem]{Remark}
\newtheorem{definition}[theorem]{Definition}
\newtheorem{conjecture}[theorem]{Conjecture}

\newtheorem{theoremletter}{Theorem}
\newtheorem{propositionletter}{Proposition}

\newcommand{\innerthmname}{}

\theoremstyle{definition}

\makeatletter
\def\namedlabel#1#2{\begingroup
	#2%
	\def\@currentlabel{#2}%
	\phantomsection\label{#1}\endgroup
}

\def\XXint#1#2#3{{\setbox0=\hbox{$#1{#2#3}{\int}$ }
		\vcenter{\hbox{$#2#3$ }}\kern-.6\wd0}}

\newcommand*\owedge{\mathpalette\@owedge\relax}
\newcommand*\@owedge[1]{%
	\mathbin{%
		\ooalign{%
			$#1\m@th\bigcirc$\cr
			\hidewidth$#1\m@th\wedge$\hidewidth\cr
		}%
	}%
}
\makeatother

\newcommand{\dive}{\mathrm{div}}

\newcommand{\ud}{\mathrm{d}}

\newcommand{\Ss}{\mathbb{S}}

\DeclareMathOperator{\Ric}{Ric}

\title[Compactness of solutions to the GJMS equation]{Compactness of singular solutions to the sixth order GJMS equation} 
\thanks{This work was partially supported by S\~ao Paulo Research Foundation (FAPESP) \#2020/07566-3 and \#2021/15139-0, Para\'iba State Research Foundation (FAPESQ) \#3034/2021, and National Council for Scientific and Technological Development (CNPq) \#312340/2021-4 and \#429285/2016-7, and Natural Sciences and Engineering Research Council of Canada (NSERC)}

\author[J.H. Andrade]{Jo\~{a}o Henrique Andrade}
\author[J.M. do \'O]{Jo\~{a}o Marcos do \'O}
\author[J. Ratzkin]{Jesse Ratzkin}
\author[J. Wei]{Juncheng Wei}

\address[J.H. Andrade]{
  Department of Mathematics,
	University of British Columbia
	\newline\indent 
	V6T 1Z2, Vancouver-BC, Canada
	\newline\indent
	and
	\newline\indent 
	Institute of Mathematics and Statistics,
	University of S\~ao Paulo
	\newline\indent 
	05508-090, S\~ao Paulo-SP, Brazil
	}
\email{\href{mailto:andradejh@math.ubc.ca}{andradejh@math.ubc.ca}}
\email{\href{mailto:andradejh@ime.usp.br}{andradejh@ime.usp.br}}

\address[J.M. do \'O]{Department of Mathematics,
	Federal University of Para\'{\i}ba
	\newline\indent 
	58051-900, Jo\~ao Pessoa-PB, Brazil}
\email{\href{mailto:jmbo@pq.cnpq.br}{jmbo@pq.cnpq.br}}

\address[J. Ratzkin]{Department of Mathematics,
	Universit\"{a}t W\"{u}rzburg
	\newline\indent
	97070, W\"{u}rzburg-BA, Germany}
\email{\href{mailto:jesse.ratzkin@mathematik.uni-wuerzburg.de}{jesse.ratzkin@mathematik.uni-wuerzburg.de}}

\address[J. Wei]{
	Department of Mathematics,
	University of British Columbia
	\newline\indent 
	V6T 1Z2, Vancouver-BC, Canada}
\email{\href{mailto:jcwei@math.ubc.ca}{jcwei@math.ubc.ca}}


\begin{document}
	
	\begin{abstract}
		We study compactness properties of the set of conformally flat singular metrics with
		constant, positive sixth order $Q$-curvature on a finitely punctured sphere.
		Based on a recent classification of the local asymptotic behavior near isolated singularities, we 
		introduce a notion of necksize for these metrics in our moduli space, which we use to characterize compactness.
		More precisely, we prove that if the punctures remain separated and the necksize at each puncture is 
		bounded away from zero along a sequence of metrics, then 
		a subsequence converges with respect to the Gromov--Hausdorff metric.
		Our proof relies on an upper bound estimate which is proved using moving planes and a blow-up argument. 
		This is combined with a lower bound estimate which is a consequence of a removable singularity theorem.
		We also introduce a homological invariant which may be of independent interest for upcoming research.
	\end{abstract}
	
	\maketitle


    \numberwithin{equation}{section} 
	\numberwithin{theorem}{section}
	
	\section{Introduction}
	In recent years, there has been active research into analogs of the Yamabe problem and 
	its singular counterpart. 
    In each of these problems, one seeks a representative 
	of a conformal class with constant curvature of some type, scalar curvature 
	in the classical case, and some $\sigma_k$-curvature or one of Branson's $Q^{2m}$-curvatures 
	in more modern examples. Conformal invariance (or, more generally, covariance) often 
	complicates these problems, leading to singular 
	solutions and the lack of compactness in the space of solutions. For this 
	reason, it is always appealing to characterize which geometric properties 
	in the solution space imply compactness. 
	
	In the present paper, we study the moduli space of complete, conformally 
	flat metrics with constant sixth order $Q^6$-curvature on a finitely punctured 
	sphere. Our main result generalizes a theorem of Pollack \cite{MR1266101} in the scalar
	curvature setting, stating that so long as the punctures remain 
	separated and certain geometric necksizes bounded away from zero, the 
	the corresponding subset of moduli space is compact in the Gromov-Hausdorff 
	topology. 
	
	Let $n \geqslant 7$ and denote the $n$-dimensional sphere by $\mathbb{S}^n$. For 
	$N \in \mathbb{N}$ we let $\Lambda = \{ p_1, \dots, p_N\} \subset \mathbb{S}^n$ 
	be a finite subset and seek complete metrics on $\Omega := \mathbb{S}^n \backslash 
	\Lambda$ of the form $g = U^{{4}/{n-6}} g_0$, where $g_0$ is the standard
	round metric. The fact that $g$ is complete on $\Omega$ forces 
	$\liminf_{p \rightarrow p_i} U(p) = \infty$ for each $i = 1, \dots, N$. Furthermore,
	we prescribe the resulting metric to have constant $Q^6$-curvature, which we normalize to 
	be 
	\begin{equation}\label{const_q_curv}
	Q_n = Q^6(g_0) = \frac{n(n^4-20n^2+64)}{2^5}.\end{equation}
	We define $Q^6(g)$ the quantity in Definition~\ref{def:curvatures} for any smooth metric. 
	
	The $Q$-curvature $Q^6$ behaves well under a conformal change of metric. More 
	precisely, the condition that $g = U^{{4}/{(n-6)}} g_0$ satisfies $Q^6(g) = Q_n$ on 
	$\Omega = \Ss^n \backslash \Lambda$ is equivalent to the PDE 
		\begin{equation}\tag{$\mathcal{Q}_{6,g_0,N}$}\label{ourequation}
        	P^6_{g_0}U=c_n U^{\frac{n+6}{n-6}} \quad \mbox{on} \quad \Omega,
    	\end{equation}
    where $c_n=\frac{n-6}{2}{Q}_n$ is a normalizing constant.
    The operator on the left-hand side is the sixth order GJMS operator on the sphere defined by 
    \begin{equation}\label{sphericalGJMSoperator}
		    P_{g_0}^6=\left(-\Delta_{g_0}+\frac{(n-6)(n+4)}{4}\right)
		    \left(-\Delta_{g_0}+\frac{(n-4)(n+2)}{4}\right)
		    \left(-\Delta_{g_0}+\frac{n(n-2)}{4}\right) ,
		\end{equation}
    and after a conformal change of metric $g = U^{{4}/{n-6}} g_0$, it transforms as 
	\begin{equation}\label{transformationlawpaneitz}
        P^6_{g}\phi=U^{-\frac{n+6}{n-6}}P^6_{g_0}(U\phi) 
    \quad \mbox{for all} \quad \phi\in \mathcal{C}^{\infty}(\Omega).
	\end{equation}
	For more details on this subject, we refer the interested reader to \cite{MR3077914,MR3652455,MR4285731,MR3073887}.
	
	In \cite{GJMS} Graham, Jenne, Mason and Sparling constructed conformally 
	covariant differential operators $P_{g}^{2m}$ on a compact $n$-dimensional Riemannian manifold 
	$(M^n,g)$ for any  
	$m \in \mathbb{N}$ such the leading order term of $P_{g}^{2m}$ is 
	$(-\Delta_{g})^m$. One can then construct the associated $Q$-curvature 
	of order $2m$ by $Q_{g}^{2m} = P_{g_0}^{2m} (1)$. In the special case $m=1$,
	one recovers the conformal Laplacian 
	$$P_g^2 = -\Delta_g + \frac{n-2}{4(n-1)} R_g \quad {\rm with} \quad Q_g^2 =\frac{n-2}{4(n-1)} R_g,$$
	where $\Delta_g$ is the Laplace-Beltrami operator of $g$ and $R_g$ is 
	its scalar curvature. 
    Subsequently, Grahan and Zworski \cite{graham-zworski} and Chang and 
    Gonz\'alez \cite{chang-gonzalez} 
    extended these definitions in the case the background metric is the 
    round metric on the sphere to obtain (nonlocal) operators $P_{g_0}^\sigma$ of any 
    order $\sigma\in (0,n/2)$ as well as the corresponding $Q$-curvatures of order 
    $\sigma$. Once again, the leading order part of $P_{g_0}^\sigma$ is $(-\Delta_{g_0})^\sigma$, 
    understood as the principal value of a singular integral operator. 
    We write 
    the formulae for $P_g^2$, $P_g^4$ and $P_g^6$ explicitly in
    Definitions~\ref{def:geometrictensors} and~\ref{def:conformaloperator}.
    Nevertheless, 
    the expressions for $P_g^\sigma$ and $Q_g^\sigma$ for a general $\sigma\in\mathbb R_+$ are 
    more complicated (see for instance \cite{MR3694655}). 
	
	We remark that the nonlinearity 
    on the right-hand side of \eqref{ourequation} has critical growth with respect 
    to the Sobolev embedding $W^{3,2} (\mathbb{R}^n) \hookrightarrow L^{2^{\#}} 
    (\mathbb{R}^n)$, where $2^\# = \frac{2n}{n-6}$. It is well known that this 
	embedding is not compact, reflecting the conformal invariance of the PDE \eqref{ourequation}. 
	
	It will be convenient to transfer the PDE \eqref{ourequation} to Euclidean 
	space, which we can do using the standard stereographic projection (with the 
	north pole in $\Omega$, and thus a regular point of any of the 
	metrics we consider). After stereographic projection, we can write 
	$$g_0 = u_{\rm sph}^{\frac{4}{n-6}} \delta, \qquad u_{\rm sph}
	(x) = \left ( \frac{1+|x|^2}{2} \right )^{\frac{6-n}{2}},$$
	where $\delta$ is the Euclidean metric. 
	In these coordinates our conformal metric takes the 
	form $g = U^{{4}/{(n-6)}}g_0 = (U\cdot u_{\rm sph})^{{4}/{(n-6)}}
	\delta$. 
    Thus, $u\in \mathcal{C}^{\infty}(\mathbb R^n\setminus\Gamma)$ given by $u = U \cdot u_{\rm sph}$ 
    is a positive singular solution to the transformed equation 
	\begin{flalign}\label{limitequationmanypunct}\tag{$\mathcal{Q}_{6,\delta,N}$}
           (-\Delta)^3u=c_{n}u^{\frac{n+6}{n-6}} \quad {\rm in} \quad \mathbb R^n\setminus\Gamma,
	\end{flalign}
    where $\Delta$ is the usual flat Laplacian and $\Gamma$ is the image of the singular 
    set $\Lambda$ under the stereographic projection. As a notational shorthand, we adopt the 
	convention that $U$ refers to a conformal factor relating the metric $g$ to the 
	round metric, {\it i.e.} $g = U^{{4}/{(n-6)}} g_0$, while $u$ refers to a conformal 
	factor relating the metric $g$ to the Euclidean metric, {\it i.e.} $g = u^{{4}/{(n-6)}}
	\delta$, with the two related by $u = U u_{\rm sph}$. 
	
	\begin{remark}\label{rmk:scalinglaw}
    In this Euclidean setting, the transformation law \eqref{transformationlawpaneitz} in particular 
    implies the scaling law for \eqref{limitequationmanypunct}, namely if $u$ solves 
    \eqref{limitequationmanypunct} then so does $u_{\lambda}(x):=\lambda^{\frac{n-6}{2}}u(\lambda x)$
    for any $\lambda >0$. 
    \end{remark}
    
	We study the compactness properties of both the unmarked 
	and the marked moduli spaces of admissible constant sixth $Q$-curvature metrics. We 
	define the unmarked moduli space as 
	\begin{align}\label{unmarkedmodulispace}
		\mathcal{M}^6_{N}=\left\{g\in [g_0] :  \mbox{$g$ is complete on $\mathbb S^n\setminus\Lambda$ 
		with $\#\Lambda=N$}, \;  Q^6_g\equiv Q_n   \right\},
	\end{align}
	and the marked moduli space as 
	\begin{align*}
		\mathcal{M}^6_{\Lambda}=\left\{g\in [g_0] :  \mbox{$g$ is complete on $\mathbb S^n\setminus\Lambda$}, 
		\;  Q^6_g\equiv Q_n   \right\}.\	\end{align*}
	Intuitively, in the unmarked moduli space we fix only the number of punctures, whereas 
	in the marked moduli space, we fix the punctures themselves. We place the Gromov-Hausdorff topology 
	on both the marked and unmarked moduli spaces. 

	The first step to understanding the properties of the marked moduli space $\mathcal{M}^6_{N}$ is 
	to study  the conformally flat equation  
	 \begin{flalign}\tag{$\mathcal P_{6,R}$}\label{ourlocalPDE}
		        (-\Delta)^3u=c_{n}u^{\frac{n+6}{n-6}} \quad {\rm in} \quad \mathbb{B}^*_R,
	\end{flalign}
	where $\mathbb{B}_R^*:=\{x\in\mathbb{R}^n : 0<|x|<R\}$ is the punctured ball for $R<+\infty$. Allowing 
	$R\rightarrow+\infty$ turns \eqref{ourlocalPDE} into the following PDE on the punctured space
     \begin{flalign}\tag{$\mathcal P_{6,\infty}$}\label{ourlimitPDE}
		        (-\Delta)^3u=c_{n}u^{\frac{n+6}{n-6}} \quad {\rm in} \quad \mathbb R^n\setminus\{0\}.
	\end{flalign}
 
	On this subject, the classification of non-singular solutions to \eqref{ourlimitPDE} is provided in \cite{MR1679783}.
    Later on, in \cite{arxiv:1901.01678} it is proved that blow-up limit solutions do exist.
    Recently, based on a topological shooting method, the first and last authors classified 
    all possible solutions to this limit equation \cite{arXiv:2210.04376}.
    
    One can merge these classification	results into the statement below
				\begin{theoremletter}\label{thmA}
				Let $u$ be a positive solution to \eqref{ourlimitPDE}. 
				Assume that
				\begin{enumerate}
						\item[{\rm (a)}] the origin is a removable singularity, then there exists $x_0\in\mathbb{R}^n$ and $\varepsilon>0$ such that $u$ is radially symmetric about $x_0$ and, up to a constant, is given by 
								\begin{equation}\label{sphericalsolutions}
									u_{x_0,\varepsilon}(x)=\left(\frac{2\varepsilon}{1+\varepsilon^{2}|x-x_0|^{2}}\right)^{\frac{n-6}{2}}.
									\end{equation}
								These are called the {\it $($sixth order$)$ spherical solutions} $($or bubbles$)$.
						\item[{\rm (b)}]  the origin is a non-removable singularity, then $u$ is radially symmetric about the origin. Moreover, there exist $\varepsilon_0 \in (0,\varepsilon^*_n]$ and $T\in (0,T_{\varepsilon_0}]$ such that
								\begin{equation}\label{eq:fourthorderemdenfwoler}
									u_{\varepsilon,T}(x)=|x|^{\frac{6-n}{2}}v_{\varepsilon}(\ln|x|+T).
									\end{equation}
								Here $\varepsilon^*_n=K_0^{(n-6)/6}$, $T_\varepsilon\in\mathbb{R}$ is the fundamental period of the unique $T$-periodic bounded solution $v_T$ to the following sixth order IVP, 
								\begin{equation*}
									\begin{cases}
										v^{(6)}-K_4v^{(4)}+K_2v^{(2)}-K_0v=c_nv^{\frac{n+6}{n-6}}\\
										v(0)=\varepsilon_0,\ v^{(2)}(0)=\varepsilon_2,\ v^{(4)}(0)=\varepsilon_4,\ v^{(1)}(0)=v^{(3)}(0)=v^{(5)}(0)=0,
										\end{cases}
									\end{equation*}
								where $K_4,K_2,K_0,\varepsilon^*_n$ are dimensional constants $\varepsilon_0\in (0,\varepsilon^*_n]$ $($See \eqref{coefficients}$)$. 
								These are called $($sixth order$)$ Emden--Fowler solutions.
					\end{enumerate}  
			\end{theoremletter}
            
             In \cite{arxiv:1901.01678}, it is shown that solutions to \eqref{ourlocalPDE} with $R<+\infty$ satisfy a priori bound near the isolated singularity, which implies that they behave like the solutions to the limit equation near the isolated singularity

       \begin{theoremletter}\label{thmB}
				Let $u$ be a positive singular solution to \eqref{ourlocalPDE}.
				Suppose that $-\Delta u\geqslant 0$ and $\Delta^2 u\geqslant 0$.  
				Then
						\begin{equation}\label{asymptotics}
							u(x)=(1+\mathrm{o}(1))u_{\varepsilon,T}(|x|) \quad {\rm as} \quad x\rightarrow0,
							\end{equation}    
                where $u_{\varepsilon,T}$ belongs to the family \eqref{eq:fourthorderemdenfwoler}.
			\end{theoremletter}
   
			These two results combined motivate the following definition
			\begin{definition}
			    Let $g\in \mathcal{M}_N$ with a singular set $\Lambda\subset 
			    \mathbb{S}^n$, $\# \Lambda = N$, and let $p_j\in \Lambda$. Let $g = U^{{4}/{(n-6)}} g_0 
			    = u^{{4}/{(n-6)}} \delta$
			    where we choose stereographic coordinates centered at $p_j$. By \eqref{asymptotics} 
			    we know $u(x) = u_{\varepsilon_j, T_j} (|x|)(1+\mathrm{o}(|x|))$ for some 
			    $\varepsilon_j \in (0,\varepsilon_n^*]$. This $\varepsilon_j$ is the 
			    asymptotic necksize of the metric $g$ at the puncture $p_j$. 
			\end{definition}
	
	Now we have conditions to state our main compactness theorem for the unmarked moduli space 
    \begin{theorem}\label{maintheorem} 
        Let $N\geqslant 3$ and let $0<\delta_1,\delta_2<1$ be positive real numbers. 
        Then the set 
        \begin{equation*}
            \mathcal{Q}^6_{\delta_1, \delta_2} = \left\{ g \in \mathcal{M}^6_N : 
        \ud_{g_0} (p_j, p_\ell) \geqslant \delta_1 \;
        {\rm for \ each} \; j \neq \ell \; {\rm and} \;\varepsilon_j(g) \geqslant \delta_2 \right\}.
        \end{equation*}
        is sequentially compact with respect to the Gromov--Hausdorff topology. 
    \end{theorem} 
    
    \begin{remark}
        Notice that as a consequence of Theorem~\ref{thmA} {\rm (a)}, it follows that $\mathcal{M}_1=\varnothing$.
        Also, from Theorem~\ref{thmA} {\rm (b)}, we have that $\mathcal{M}_{p_1,p_2}
        = (0,\varepsilon_n^*]$ for any $p_1\neq p_2$, where $\varepsilon_n^*\in(0,1)$.
        Moreover, it follows that $\mathcal{M}_2 = (0,\varepsilon_n^*] \times ((\mathbb S^n \times 
        \mathbb S^n \setminus{\rm diag})/SO(n+1,1))$, where the group $SO(n+1,1)$ of conformal 
        transformations acts on each $\mathbb S^n$ factor simultaneously. These metrics are called the 
        Delaunay metrics.
        Furthermore, they all correspond to a doubly punctured sphere and are rotationally invariant.  
    \end{remark}
    
    \begin{remark} It is worthwhile to now describe the possible degenerations of a 
    sequence of metrics in $\mathcal{M}_N^6$. 
    Let $\{ g_k = (U_k)^{{4}/{n-6}} g_0 \} 
    \in \mathcal{M}_N^6$ be a sequence that leaves every compact subset. We denote 
    the singular set of $g_k$ by $\Lambda_k = \{ p_{1,k}, \dots, p_{N,k}\}$ 
    and the asymptotic necksize of $g_k$ at the puncture $p_{j,k}$ as $\varepsilon_{j,k}$. 
    Then either $\lim_{k\rightarrow \infty} \varepsilon_{j,k} = 0$ for some $j$ 
    or $\lim_{k \rightarrow \infty} p_{j,k} = \lim_{k \rightarrow \infty} p_{j', k}$ 
    for some $j \neq j'$. 
    We sketch these two degenerations in Figure~\ref{degenerations_fig}. 
    $($It is possible that both degenerations happen simultaneously.$)$ 
    In either 
    case, in the limit one obtains a metric $g_\infty \in \mathcal{M}_{N'}^6$ for some 
    $N' < N$. 
    In this way, one can compactify the moduli space $\mathcal{M}_N^6$ by 
    gluing copies of $\mathcal{M}_{N'}^6$ for $N'<N$ to $\partial \mathcal{M}_N^6$. 
    We speculate that this compactification would not give a smooth manifold with 
    boundary, but rather that $\partial \mathcal{M}_N^6$ is in general a stratified space. 
    \end{remark} 
    
    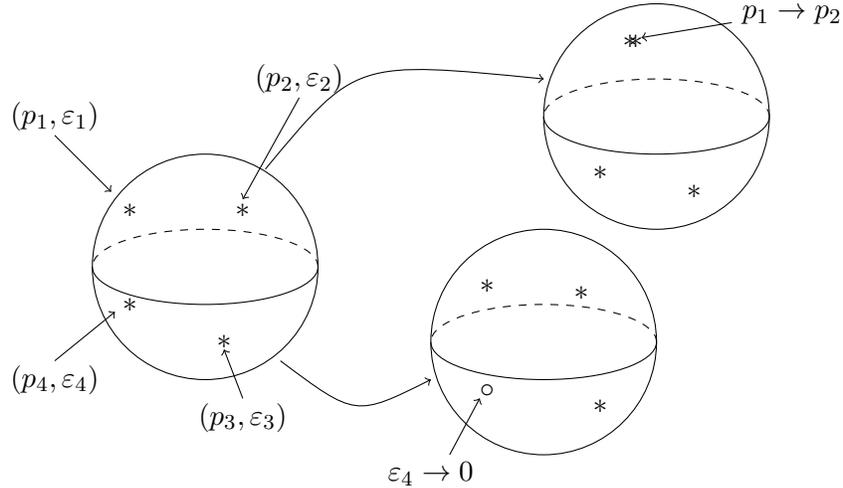
\begin{figure} [h]
    \centering
    \begin{tikzpicture}
    \draw (-2,0) circle (1.5cm);
    \draw [dashed] (-0.5,0) arc (0:180:1.5 and 0.5); 
    \draw (-3.5,0) arc (-180:0:1.5 and 0.5); 
    \node at (-3,0.75) {$*$}; 
    \draw [->] (-4,1.75) -- (-3.25,1); 
    \node at (-4, 2) {$(p_1, \varepsilon_1)$}; 
    \node at (-3, -0.5) {$*$}; 
    \draw [->] (-4,-1.25) -- (-3.2,-0.6);
    \node at (-4, -1.5) {$(p_4, \varepsilon_4)$};
    \node at (-1.5,0.75) {$*$}; 
    \draw [->] (-0.75,2.25) -- (-1.5,0.9);
    \node at (-0.75,2.5) {$(p_2,\varepsilon_2)$}; 
    \node at (-1.75, -1) {$*$}; 
    \draw [->] (-1.5,-1.75) -- (-1.75,-1.1); 
    \node at (-1.5,-2) {$(p_3, \varepsilon_3)$}; 
    \draw (4,2) circle (1.5cm); 
    \draw [dashed] (5.5,2) arc (0:180:1.5 and 0.5); 
    \draw (2.5,2) arc (-180:0:1.5 and 0.5); 
    \node at (3.65,3.0) {$*$}; 
    \node at (3.725, 3.0) {$*$}; 
    \draw [->] (5,3.25) -- (3.8, 3.05); 
    \node at (5.80,3.35) {$p_1 \rightarrow p_2$};
    \node at (4.5, 1.0) {$*$}; 
    \node at (3.25, 1.25) {$*$}; 
    \draw (2.5,-1) circle (1.5cm); 
    \draw [dashed] (4,-1) arc (0:180: 1.5 and 0.5); 
    \draw (1,-1) arc (-180:0:1.5 and 0.5); 
    \node at (1.75,-0.25) {$*$}; 
    \node at (3.0, -0.35) {$*$}; 
    \node at (3.25, -1.85) {$*$}; 
    \node at (1.75, -1.65) {$\circ$};  
    \draw [->] (1.25, -2.5) -- (1.65,-1.75); 
    \node at (1.0, -2.75) {$\varepsilon_4 \rightarrow 0$}; 
    \draw [->] (-1,-1.25) .. controls (0,-2) .. (1,-1.5); 
    \draw [->] (-1.2, 1.3) ..  controls (0,2.75) ..  (2.5,2.5); 
    \end{tikzpicture}
    \caption{The two possible degenerations in the moduli space $\mathcal{M}_4^6$.} 
    \label{degenerations_fig}
    \end{figure}
    
	Let us compare our main results with the second and fourth order analogs.
	In the same spirit as our main result, it was proved in \cite{MR1266101} and \cite{10.1093/imrn/rnab306} that the moduli sets below
	\begin{align}\label{modulispace2th}
    \mathcal{Q}^2_{\delta_1, \delta_2}\subset{\mathcal{M}}^2_{N}=\left\{g\in [g_0] : \mbox{$g$ is complete}  \; \mbox{and} \; R_g\equiv 2^{-1}(n-4) \right\}.
	\end{align}
	and 
	\begin{align}\label{modulispace4th}
	\mathcal{Q}^4_{\delta_1, \delta_2}\subset{\mathcal{M}}^4_{N}=\left\{g\in [g_0] : \mbox{g is complete},  \; R_g\geqslant 0, \; \mbox{and} \; Q^4_g\equiv 2^{-3}n(n^2-4) \right\}.
	\end{align}
	are also sequentially compact.

    Based on classifications results like Theorem~\ref{thmA} and Theorem~\ref{thmB}, much more is known about the moduli spaces in \eqref{modulispace2th} and \eqref{modulispace4th}. 
    In fact, in some classical works of Mazzeo and Pacard \cite{MR1425579} used gluing techniques to prove that there exists a family of solutions in \eqref{modulispace2th}.
    Furthermore, Mazzeo, Pollack, and Uhlenbeck \cite{MR1356375}, this space turns out to be a finite-dimensional analytic submanifold furnished with a natural Lagrangian structure.
    On the moduli space \eqref{modulispace4th}, less is known; it is not proved yet whether this is non-empty. 
    Some of the authors in \cite{arXiv:2110.05234} proved that this property holds for non-degenerate manifolds with a suitable hypothesis on the vanishing of the Weyl tensor, However, the standard round sphere is not included in this class.
	
 	Inspired by the arguments in \cite{MR1266101}, the proof of Theorem~\ref{maintheorem} is divided into three 
 	parts that we describe as follows.	
	First, we need to introduce the so-called {sixth order geometric Pohozaev invariant}, which is related to 
	the Hamiltonian energy of the limiting ODE \cite{MR929283,MR0192184}. 
    Second, we obtain an {\it a priori} upper and for positive singular solutions to \eqref{ourequation}, 
    estimates which are accomplished by combining a sliding method, a blow-up argument, and a Harnack inequality.
    From this, we obtain uniform bound on certain H\"{o}lder norms, which by compactness, allows us to 
    extract a limit, up to subsequence.
    Third, we use the first order asymptotic expansion for the Green function of the sixth order GJMS operator 
    near the pole and the fact the necksizes are away from zero shows that this limit is non-trivial.
    Finally, one can apply a removable singularity theorem to conclude the proof.

    The rest of the paper is divided as follows. 
	In Section~\ref{sec:cylcoordinates}, we define the logarithmic cylindrical change of variables 
	and we use the conformal invariance between the punctured space and the cylinder to 
	transform \eqref{ourequation} into a PDE on the cylinder. 
	In Section~\ref{sec:delaunaymetrics}, we describe all singular solutions on a doubly 
	punctured sphere. These Delaunay metrics are especially important because they provide 
	asymptotic models for the metrics in $\mathcal{M}_N^6$ near a given puncture point. 
	In Section~\ref{sec:radialpohozaev}, we define the sixth order Pohozaev invariants 
	associated with \eqref{ourequation}. In Section~\ref{sec:estimates}, we prove {\it a priori} 
	upper and lower bound estimates for positive singular solutions to \eqref{ourequation}. 
	In Section~\ref{sec:mainresult}, we prove the compactness statement in Theorem~\ref{maintheorem}.
	
	\begin{remark} 
    Several of our supporting results below generalize to the Paneitz operators 
    and $Q$-curvatures of any order $\sigma \in (0,n/2)$, at least in the conformally 
    flat setting. In particular, the convexity result of Lemma~\ref{lm:geodesicallyconvex} 
    and the upper bound of Proposition~\ref{prop:upperestimate} both generalize, and may 
    be of independent interest. On the other hand, some parts of the proof of Theorem~\ref{maintheorem}
    do not carry over. In particular, at this time we cannot classify all two-ended constant 
    $Q^\sigma$-curvature metrics on the sphere, which is very important for our proof. 
    \end{remark}

	\section{Cylindrical coordinates}\label{sec:cylcoordinates}
	This section is devoted to constructing a change of variables that transforms the local 
	singular PDE \eqref{ourlocalPDE} problem into a nice ODE problem with constant coefficients.
	This is the conformally flat problem associated with \eqref{ourequation}.
	
	\begin{definition}\label{def:cylindricaltransformation}
	We define the sixth order autonomous Emden-Fowler change of variables as 
	follows. Let $R>0$ and $T= -\ln R$ and $\mathcal{C}_T = (T,\infty) \times 
	\mathbb{S}^{n-1}$. We then define 
	\begin{equation} \label{cylindricaltransformation}
	\mathfrak{F}: \mathcal{C}^\infty (B_R^*) \rightarrow \mathcal{C}^\infty
	(\mathcal{C}_T), \qquad \mathfrak{F}(u) (t,\theta) = e^{-\gamma_n t} 
	u(e^{-t} \theta) = v(t,\theta) , 
	\end{equation} 
	where $\gamma_n = \frac{n-6}{2}$. 
	\end{definition} 
	It is easy to show the inverse transform is given by 
	$$\mathfrak{F}^{-1} : \mathcal{C}^\infty(\mathcal{C}_T) 
	\rightarrow \mathcal{C}^\infty (B_R^*), \qquad \mathfrak{F}^{-1} 
	(v)(x) = |x|^{-\gamma_n} v(-\ln |x|, x/|x|) = u(x).$$ 
	
	Using $\mathfrak{F}$ and performing a lengthy computation we arrive at the following 
	sixth order nonlinear PDE on $\mathcal{C}_T$: 
	\begin{equation}\tag{$\mathcal C_{T}$}\label{ourPDEcyl}
		-P^6_{\rm cyl}v=c_{n}v^{\frac{n+6}{n-6}} \quad {\rm on} \quad {\mathcal{C}}_T.
	\end{equation}
	Here $P^6_{\rm cyl}$ is the sixth order GJMS operator associated to 
	the cylindrical metric $g_{\rm cyl} = dt^2 + d\theta^2$ on $\mathbb{R} \times 
	\Ss^{n-1}$, and it is given by
	\begin{align*}
		P^6_{\rm cyl}:=P^6_{\rm rad}+P^6_{\rm ang},
	\end{align*}
	where
	\begin{align*}
		P^6_{\rm rad}:=\partial_t^{(6)}-K_{4}\partial_t^{(4)}+K_{2}\partial_t^{(2)}-K_{0}
	\end{align*}
	and
	\begin{align*}
		P^6_{\rm ang}:=2\partial^{(4)}_t\Delta_{\theta}-J_3\partial^{(3)}_t\Delta_{\theta}+J_2\partial^{(2)}_t\Delta_{\theta}-J_1\partial_t\Delta_{\theta}+J_0\Delta_{\theta}+3\partial^{(2)}_t\Delta^2_{\theta}-L_0\Delta^2_{\theta}+\Delta^3_{\theta}
	\end{align*}
	with
	\begin{align}\label{coefficients}
		&\nonumber K_{0}=2^{-8}(n-6)^2(n-2)^2(n+2)^2&\\\nonumber 
		&K_{2}=2^{-4}(3n^4-24n^3+72n^2-96n+304)&\\\nonumber
		&K_{4}=2^{-2}(3n^2-12n+44)&\\
		&J_{0}=2^{-3}(3n^4-18n^3-192n^2+1864n-3952)&\\\nonumber
		&J_{1}=2^{-1}(3n^3+3n^2-244n+620)&\\\nonumber
		&J_{2}=2 n^2+13n-68&\\\nonumber
		&J_{3}=2 (n+1)&\\\nonumber
		&L_{0}=2^{-2}(3 n^2-12n-20)&
	\end{align}
	dimensional constants.
	
	\begin{remark}
	     The following decomposition holds
		       \begin{equation*}
		           P^6_{\rm rad}=L_{\lambda_1}\circ L_{\lambda_2}\circ L_{\lambda_3},
		       \end{equation*}
		       where $L_{\lambda_j}:=-\partial_t^2+\lambda_j$ for $j=1,2,3$ with
		       \begin{equation*}
		            \lambda_1=\frac{n-6}{2}, \quad  \lambda_2=\frac{n-2}{2}, \quad {\rm and} \quad \lambda_3=\frac{n+2}{2}.
		        \end{equation*}
                We refer the reader to \cite[Proposition~2.7]{arXiv:2210.04376} for the proof.
	\end{remark}
		
	\section{Spherical and Delaunay metrics}\label{sec:delaunaymetrics}
	In this section, we present some particular model metrics on the moduli space. 
	Let $p_1,p_2\in\mathbb S^n$, which without loss of generality can be chosen such 
	that $p_1=\mathbf{e}_n$ is the north pole and $p_2=-p_1$ is the south pole.
	The conformal factor $U:\mathbb{S}^{n} \backslash \{ p_1, p_2\} \rightarrow (0,\infty)$ 
	determines a metric $g\in\mathcal{M}_{p_1,p_2}$ and after  composing with 
	a stereographic projection it corresponds to a singular solution to \eqref{ourlimitPDE}
		
		Applying the cylindrical transform \eqref{cylindricaltransformation} to this 
		PDE in turn yields 
		\begin{equation*}
		-P^6_{\rm cyl}v=c_{n}v^{\frac{n+6}{n-6}} \quad {\rm on} \quad {\mathcal{C}}_\infty:=\mathbb R\times\mathbb S^n.
	    \end{equation*}
	    Next, using those solutions to \eqref{ourequation} are radially symmetric with respect 
	    to the origin, \eqref{ourPDEcyl} reduces to a sixth order ODE problem
	    \begin{equation}\tag{$\mathcal{O}_{6,\infty}$}\label{ourODE}
		        -v^{(6)}+K_4v^{(4)}-K_2v^{(2)}+K_0v=c_nv^{\frac{n+6}{n-6}} \quad {\rm in} \quad \mathbb R.
		\end{equation}
		From this last formulation, we quickly compute the cylindrical solution 
		$$v_{\rm cyl}(t) = \left ( \frac{K_0}{c_n} \right )^{\frac{12}{n-6}} 
		= \left ( \frac{K_0}{c_n} \right )^{\frac{6}{\gamma_n}}= \varepsilon_n^* > 0,$$ 
		which is the only constant solution. Transforming back from the cylinder 
		to $\mathbb{R}^n \backslash \{ 0 \}$ we see 
		$$u_{\rm cyl}(x) = \mathfrak{F}^{-1} (v_{\rm cyl}) = \left ( \frac{K_0}{c_n}
		\right )^{\frac{12}{n-6}} |x|^{-\gamma_n}, \quad g_{\rm cyl} 
		= u_{\rm cyl}^{\frac{4}{n-6}} \delta .$$ 
	
	We have already encountered the spherical solution, given by 	
        \begin{equation}\label{standardsphericalsolutions}
            u_{\rm sph}(x)=\left ( \frac{1+|x|^2}{2} \right )^{-\gamma_n} \quad {\rm and} 
            \quad g_{\rm sph}=u_{\rm sph}^{4/(n-6)}\delta, 
        \end{equation}
        which is the particular case of \eqref{sphericalsolutions} with $\epsilon = 1$ 
        and $x_0 = 0$. Applying the Emden-Fowler change of variables 
        to $u_{\rm sph}$ we obtain 
        $$v_{\rm sph}(t,\theta) = \mathfrak{F}(u_{\rm sph})(t,\theta) 
        = (\cosh t)^{-\gamma_n}.$$ 
    
	    In this setting,
		Theorem~\ref{thmA} classifies all positive solutions 
		$v_{\varepsilon_0} \in \mathcal{C}^{6}(\mathbb{R})$ to \eqref{ourODE} in terms of the necksize 
		$\varepsilon_0 \in(0, \varepsilon^{*}_n]$, where 
		$\varepsilon_0=\min_{\mathbb R}v \in(0, \varepsilon^{*}_n]$.
	    Varying the parameter $\varepsilon$ from its maximal value of $\varepsilon_n^*$ to $0$, 
	    we see that the Delaunay solutions in Theorem~\ref{thmA} (b) interpolate 
	    between the cylindrical solution $v_{\rm cyl}$ and the spherical solution $v_{\rm sph}$. We 
	    denote the minimal period of $v_\varepsilon$ by $T_\varepsilon$. 
        
        \begin{definition}
        For each $\varepsilon \in (0, \varepsilon_n^*]$ the Delaunay metric of 
        necksize $\varepsilon$ is 
        $$g_\varepsilon = v_\varepsilon^{\frac{4}{n-6}} (\ud t^2 + \ud\theta^2) = 
        u_\varepsilon^{\frac{4}{n-6}} \delta,$$ 
        where $u_\varepsilon = \mathfrak{F}^{-1} (v_\varepsilon)$. 
        Observe that we have equivalently defined $g_\varepsilon$ as a metric 
        on $\mathcal{C}_{-\infty}$, using $v_\varepsilon$ as the conformal factor, and
        on $\mathbb{R}^n \backslash \{ 0 \}$, using $u_\varepsilon = \mathfrak{F}^{-1} 
        (v_\varepsilon)$ as the conformal factor.  
        \end{definition}
		
		We can reformulate the expansion \eqref{asymptotics} to read 
        \begin{proposition}
        Let $g \in \mathcal{M}^6_N$ with the singular set $\Lambda$ and let 
        $p \in \Lambda$. Then there exists a Delaunay solution $u_\varepsilon$ 
        such that in stereographic coordinates centered at $p$ the asymptotic 
        expansion 
        $$g = ((1+\mathrm{o}(|x|)) u_{\varepsilon, R} (x))^{\frac{4}{n-6}} \delta, 
        \qquad u_{\varepsilon,R} (x) = u_\varepsilon (Rx).$$ 
        We can restate this asymptotic expansion as 
        \begin{eqnarray*}
        g & = & ((1+\mathrm{o}(|x|)) \mathfrak{F}^{-1}(v_\varepsilon (\cdot + T))(x))^{\frac{4}{n-6}}
        \delta=((1+\mathrm{o}(e^{-t})) v_\varepsilon(t+T))^{\frac{4}{n-6}} (\ud t^2 + \ud\theta^2).
        \end{eqnarray*}
        In other words, any admissible metric is asymptotic to a 
        translated Delaunay metric near a puncture.
        In the formulae above $R$ and $T$ are related by $R = -\ln T$. 
        \end{proposition}
        
	\section{Pohozaev invariants}\label{sec:radialpohozaev}
	
		We now turn to a discussion of the existence and specific form of a family of homological integral 
		invariants of solutions of equation \eqref{ourequation}. 
	These homological invariants were discovered in their simplest form by S. Pohozaev \cite{MR0192184}, and 
	generalized by R. Schoen \cite{MR929283} for the Riemannian setting.
	
	As a starting point, we define the energy $\mathcal{H}_{\rm cyl}$ by
    \begin{equation} \label{hamiltonian1}
        \mathcal{H}_{\rm cyl}(v):=\mathcal{H}_{\rm rad}(v)+\mathcal{H}_{\rm ang}(v)+F(v),
    \end{equation}
    where
    \begin{equation*}
        \mathcal{H}_{\rm rad}(v):=\frac{1}{2}{v^{(3)}}^2+\frac{K_4}{2}{v^{(2)}}^{2}+\frac{K_2}{2}{v^{(1)}}^2-\frac{K_0}{2}v^2
        +v^{(5)}v^{(1)}-v^{(4)}v^{(2)}-{K_4}
        v^{(3)}v^{(1)},
    \end{equation*}
    is the radial part,
    \begin{align*}
        \mathcal{H}_{\rm ang}(v)&:=-J_4\left(\partial_t^{(3)}\nabla_\theta v\partial_t\nabla_\theta v-|\partial_t^{(2)}\nabla_\theta v|^2\right)-\frac{J_2}{2}|\partial_t^{(2)}\nabla_\theta v|^2-\frac{J_1}{2}|\partial_t^{(2)}\nabla_\theta v|^2&\\
		&-\frac{J_0}{2}|\nabla_\theta v|^2+\frac{L_2}{2}|\partial_t^{(2)}\Delta_\theta v|^2+\frac{L_0}{2}|\partial_t^{(2)}\Delta_\theta v|^2+\frac{1}{2}|\Delta_{\theta}v|^2.&
    \end{align*}
    is the angular part, and
    \begin{equation*}
        F(v):=\frac{c_n(n-6)}{2n}|v|^{\frac{2n}{n-6}}
    \end{equation*}
    is the nonlinear term.
    
    Evaluating a derivative, one can easily verify $\mathcal{H}_{\rm cyl}(v)$ is 
    constant for any solution $v$ of the PDE \eqref{ourPDEcyl}. 
    We further observe that the 
    last term $F$ in \eqref{hamiltonian1} is homogeneous of degree $\frac{2n}{n-6}$ while 
    the remaining terms are all homogeneous of degree $2$. 
    
	\begin{definition}\label{def:cylidrincalpohozaev}
		Let $v\in \mathcal{C}^6(\mathcal{C}_T)$ be a positive solution to \eqref{ourPDEcyl}.
		We define its cylindrical Pohozaev invariant as
		\begin{equation*}
		\mathcal{P}_{\rm cyl}(v):=\int_{\{ t \} \times \Ss^{n-1}}\mathcal{H}_{\rm cyl}(v) \ud\theta
		\end{equation*}
		for any $t>T$. Observe that this integral does not, in fact, depend on $t$. 
	\end{definition}
	
	In light of the cylindrical transformation from Definition~\ref{def:cylindricaltransformation}, we
	can define this invariant in spherical coordinates
	\begin{definition}\label{def:radialpohozaev}
		Let $u\in \mathcal{C}^6(\mathbb{B}^*_R)$ be a positive solution to \eqref{ourlocalPDE}.
		We define its spherical Pohozaev invariant as
		\begin{equation*}
		\mathcal{P}_{\rm sph}(u):=(\mathcal{P}_{\rm cyl}\circ\mathfrak{F}^{-1})(u) 
		= \int_{\{ t \} \times \Ss^{n-1}} \mathcal{H}_{\rm cyl} (\mathfrak{F}^{-1}(u)) \ud \theta .
		\end{equation*}
	\end{definition}
	
	Finally, in terms of conformal metrics, we have the following definition of an invariant 
	associated with metrics in the moduli space.
    \begin{definition}
        Let $g \in \mathcal{M}^6_N$ and $p_j\in\Lambda$.
        We define its radial $($or dilational$)$  Pohozaev invariant at the puncture $p_j$ as
        follows. Choose stereographic coordinates sending $p_j$ to the origin and write $g = 
        u^{\frac{4}{n-6}} \delta$ in these coordinates. Then define 
		\begin{equation*}
		\mathcal{P}_{\rm rad}(g,p_j):=\mathcal{P}_{\rm sph}(u) = \int_{\{ t \} 
		\times \Ss^{n-1}} \mathcal{H}_{\rm cyl} (\mathfrak{F}^{-1} (u)) \ud \theta.  
		\end{equation*}
    \end{definition}
 
   The most important result of this section states that bounding the radial Pohozaev invariants away from zero is equivalent to bounding the necksizes of the Delaunay asymptotes away from zero. 
    \begin{proposition}\label{prop:necksizexpohozaev}
    Let $g \in \mathcal{M}^6_N$ and $p_j\in\Lambda$.
    Then $\mathcal{P}_{\textrm{rad}}(g,p_j)$ is well-defined, negative  
    and depends only on the necksize $\varepsilon_j$ of the Delaunay 
    asymptote at $p_j\in \Lambda$. 
    Moreover, decreasing $\varepsilon_j$ will increase $\mathcal{P}_{\textrm{rad}}(g,p_j)$ 
    and if $\varepsilon_j \searrow 0$ then $\mathcal{P}_{\textrm{rad}}(g,p_j) 
    \nearrow 0$. 
    \end{proposition} 
    
    \begin{proof} 
    By construction, the integral defining $\mathcal{P}_{\rm rad}(g,p_j)$ does 
    not depend on which sphere $\{ t \} \times \Ss^{n-1}$ we choose, so long 
    as $t$ is sufficiently large, and therefore $\mathcal{P}_{\rm rad}$ is 
    well-defined. By the asymptotics in Theorem \ref{thmB} we know that the 
    conformal factor is asymptotic to a Delaunay solution $u_\varepsilon$, and 
    so letting $t \rightarrow \infty$ we see 
    $$\mathcal{P}_{\rm rad} (g,p_j) = \lim_{t \rightarrow \infty} 
    \int_{\{ t \} \times \Ss^{n-1}} \mathcal{H}_{\rm cyl} (\mathfrak{F}^{-1} (u))
    \ud \theta = \lim_{t \rightarrow \infty} \int_{\{ t \} \times \Ss^{n-1}} 
    \mathcal{H}_{\rm cyl} (v_\varepsilon) \ud \theta 
    < 0.$$ The remaining properties follow directly from energy ordering of the 
    Delaunay solutions as described in 
   \cite[Lemma~4.14]{arXiv:2210.04376}. 
    \end{proof}

     \begin{remark}
        One often finds integral invariants in geometric variational problems. 
        For more details on a class of general higher order conformally invariant locally conserved tensors, we cite \cite{MR3073449}.
        These invariants arise from the conformal invariance of \eqref{ourequation}, by Noether's famous conservation theorem.
    \end{remark}
    
    For our later applications we will need a slight refinement of 
    Proposition \ref{prop:necksizexpohozaev}. 
    \begin{proposition} \label{prop:rescaled_necksize_poho}
    Let $v\in \mathcal{C}^6(\mathcal{C}_T)$ be a positive solution to the following rescaled equation
    $$-P_{\rm cyl}^6v = A v^{\frac{n+6}{n-6}}$$ 
    for some constant $A$ and let 
    $$\mathcal{H}_{\rm cyl}^A (v) = \mathcal{H}_{\rm rad}(v) + \mathcal{H}_{\rm ang}(v) 
    + \frac{(n-6)A}{2n} |v|^{\frac{2n}{n-6}}.$$ 
    Then 
    $$\int_{\{ t \} \times \Ss^{n-1}} \mathcal{H}_{\rm cyl}^A (v) \ud \theta$$ 
    is independent of $t$. 
    \end{proposition} 
    
    \begin{proof} The proposition follows from taking the derivative with respect to 
    $t$ and integrating by parts. 
    \end{proof} 
	
	\section{Uniform estimates}\label{sec:estimates}
	This section is devoted to proving uniform upper and lower estimates near 
	the singular set for positive singular solutions to \eqref{ourequation}.
    
    We begin by quoting a superharmonicity result of Ng\^{o} and Ye \cite{MR4438901}. 
    We also remark a similar superharmonicity result for a related integral equation Ao et al. \cite{MR4420104}.
    
    \begin{propositionletter}\label{prop:super_harm}
    Let $u\in \mathcal{C}^{\infty}(\mathbb{R}^n \setminus \Gamma)$ be a positive solution to \eqref{limitequationmanypunct}. 
    Then additionally 
    $-\Delta u \geqslant 0$ and $\Delta^2 u \geqslant 0$ in $\mathbb{R}^n \setminus \Gamma$. 
    \end{propositionletter} 
    
    \begin{proof} Following  \cite[Proposition 1.5]{MR4438901} we see 
    that $u$ is both weakly superharmonic and weakly 
    superbiharmonic in $\mathbb{R}^n$. 
    In other words, for a
    smooth test function $\phi$ compactly supported in $\mathbb{R}^n 
    \setminus \Gamma$, we have 
    \begin{equation*}
        \int_{\mathbb{R}^n} u(-\Delta) \phi \ud x\geqslant 0 \quad {\rm and} \quad \int_{\mathbb{R}^n} u(-\Delta)^2 \phi \ud x \geqslant 0.
    \end{equation*}
    Standard elliptic regularity then implies $u$ is superharmonic 
    and superbiharmonic where it is smooth, namely in $\mathbb{R}^n 
    \setminus \Gamma$. 
    \end{proof} 
     
    The first step is a sixth order version of the convexity result \cite[Proposition~1]{schoen-convexity}, which 
    is proved using the Alexandrov's moving  planes (see also \cite[Theorem~4.1]{chang-han-yang} for a 
    fourth order version). 
	
	\begin{lemma}\label{lm:geodesicallyconvex}
	Let $g = U^{{4}/{(n-6)}} g_0$ be a complete metric on $\Omega = \mathbb{S}^n \backslash \Lambda$ 
	which is conformal to the round metric, such that $Q_g^6$ is a positive constant. 
	Then the boundary of any $($spherically$)$ round ball in $\Omega$ has a non-negative 
    definite second fundamental form with respect to $g$. 
	\end{lemma}
	
	\begin{proof} We let $\mathcal{B}$ be a geodesic ball with respect to the 
	round metric such that $\overline{\mathcal{B}} \subset \Omega$ and 
	choose a stereographic projection that sends $\mathcal{B}$ to the 
	half-space $\{ x\in \mathbb{R}^n : x_1 < 0 \}$. As before, we denote 
	the image of the singular set $\Lambda$ under this stereographic 
	projection by $\Gamma$. With respect to these stereographic coordinates 
	the metric takes the form $g = u^{{4}/{(n-6)}} \delta$ where $u\in \mathcal{C}^{\infty}(\mathbb{R}^n \setminus \Gamma)$ 
	satisfies \eqref{limitequationmanypunct}, namely $u:\mathbb{R}^n \setminus \Gamma \rightarrow (0,\infty)$ satisfy
    \begin{equation*}
        (-\Delta)^3 u = c_n u^{\frac{n+6}{n-6}} \quad {\rm in} \quad \mathbb{R}^n \setminus \Gamma.
    \end{equation*}
    Furthermore, the boundary of our round ball is 
    $\partial \mathcal{B} = \{ x \in \mathbb{R}^n : x_1 = 0\}$ 
    and is oriented by the inward unit normal $\eta_g = u^{-2/(n-6)} 
    \partial_{x_1}$. It follows that the second fundamental form $II$ and 
    mean curvature $H$ of $\partial \mathcal{B}$ are given by 
    $$II_{ij} = -\langle \nabla_{\partial x_j} \eta, \partial_{x_i} \rangle 
    = \frac{2}{n-6} \delta_{ij} u^{\frac{8-n}{n-6}} \partial_{x_1} u, \quad 
    H = \frac{2n}{n-6} u^{\frac{8-n}{n-6}} \partial_{x_1} u.$$ 
    Therefore, (weak) convexity of $\partial \mathcal{B}$ follows once we show 
    $\partial_{x_1} u \geqslant 0$ along the hyperplane $\{ x_1 = 0\}$. 
    
    By Proposition \ref{prop:super_harm}, we have 
    \begin{equation*}
        -\Delta u  \geqslant 0 \quad {\rm and} \quad (-\Delta)^2 u \geqslant 0 \quad {\rm in} \quad \mathbb R^n\setminus\Gamma.
    \end{equation*}
    We now rewrite \eqref{limitequationmanypunct} as a second order system, 
    letting 
    \begin{equation*}
        \mathrm u_0 = u, \quad \mathrm u_1 = -\Delta u \quad {\rm and} \quad \mathrm u_2 = (-\Delta)^2 u.
    \end{equation*}
    so that we obtain $\mathrm u_i : \mathbb{R}^n \setminus \Gamma \rightarrow (0,\infty)$ for $i=0,1,2$ satisfy
    \begin{equation} \label{second_order_system} 
    \left \{ \begin{array}{rcl} -\Delta \mathrm u_0 & = & \mathrm u_1 \geqslant 0 \\ 
    -\Delta \mathrm u_1 & = & \mathrm u_2 \geqslant 0\\ -\Delta \mathrm u_2 & = & c_n u_0^{\frac{n+6}{n-6}} \geqslant 0 . 
    \end{array} \right. 
    \end{equation} 
   
    It follows from \cite[Theorem~2.7]{MR931204} that the Newtonian capacity of the singular set vanishes, 
    {\it i.e.} $\operatorname{cap}(\Gamma)=0$.
    As a consequence, one can find $a_0>0$ and $a_j\in\mathbb{R}$ for $j=1,\dots,n$ such that
     \begin{align}\label{asymp1}
        \begin{cases}
            \mathrm u_0(x) &=a_0|x|^{6-n}+\sum_{j=1}^n {a_j x_j}{|x|^{4-n}}+\mathcal{O}\left({|x|^{4-n}}\right)\\
            \mathrm \partial_{x_i}u_0(x) &=-(n-6) a_0 x_i|x|^{4-n}+\mathcal{O}\left(|x|^{4-n}\right)\\
            \mathrm \partial^2_{x_i x_j} u_0(x) &=\mathcal{O}\left({|x|^{4-n}}\right),
        \end{cases}
    \end{align}
    which, by differentiating further, yields 
    \begin{align}\label{asymp2}
        \begin{cases}
            \mathrm u_1(x) &=b_0|x|^{4-n}+\sum_{j=1}^n {b_j x_j}{|x|^{2-n}}+\mathcal{O}\left({|x|^{2-n}}\right) \\
            \partial_{x_i}  \mathrm u_1(x)&=-(n-4) b_0 x_i|x|^{-n}+\mathcal{O}\left({|x|^{2-n}}\right) \\
            \partial^{(2)}_{x_i x_j} \mathrm u_{1}(x) &=\mathcal{O}\left({|x|^{2-n}}\right)
        \end{cases}    
    \end{align}
    and 
    \begin{align}\label{asymp3}
        \begin{cases}
            \mathrm u_2(x) &=c_0|x|^{2-n}+\sum_{j=1}^n {c_j x_j}{|x|^{-n}}+\mathcal{O}\left({|x|^{-n}}\right) \\
            \partial_{x_i} \mathrm u_2(x)&=-(n-2) c_0 x_i|x|^{-n}+\mathcal{O}\left({|x|^{2-n}}\right) \\
            \partial^{(2)}_{x_i x_j} \mathrm u_2 (x) &=\mathcal{O}\left({|x|^{-n}}\right)
        \end{cases}    
    \end{align}
    as $|x|\rightarrow0$, where $b_0, c_0 > 0$ and $b_j, c_j \in \mathbb{R}$ for $j=1,\dots, n$. 
   
    We are now ready to set up the method of moving planes applied to the triple of 
    functions $(\mathrm u_0, \mathrm u_1, \mathrm u_2)$. 
    For any $\lambda \in \mathbb{R}$, we let $\Sigma_\lambda = \{ x \in \mathbb{R}^n : x_1 > \lambda\}$ and 
    $T_\lambda = \partial \Sigma_\lambda = \{ x \in \mathbb{R}^n : x_1 = \lambda\}$. We also set
    $\Sigma_\lambda^{\prime} = \Sigma_\lambda \setminus \Gamma$. For any $x \in \Sigma_\lambda^{\prime}$,
    we let 
    $$x^{\lambda} = (2\lambda-x_1, x_2, \dots, x_n)$$ 
    be the reflection of $x$ across the hyperplane $T_\lambda = \{ x_1 = \lambda\}$. Finally, 
    our goal in moving planes is to show that for any $\lambda \leqslant 0$ and $i=0,1,2$, we have 
    \begin{equation} \label{goal_moving_planes} 
    \mathrm w_i^\lambda (x)>0 \quad {\rm for} \quad i=0,1,2,
    \end{equation}
    where $\mathrm w_i^\lambda : \Sigma_\lambda^{\prime} \rightarrow \mathbb{R}$ is given by
    \begin{equation*}
    \mathrm w_i^\lambda (x) 
    = \mathrm u_i(x) - \mathrm u_i(x^\lambda). 
    \end{equation*}
    Once we establish \eqref{goal_moving_planes}, letting $\lambda \nearrow 0$ 
    the first inequality implies $\mathrm u_{x_1} \geqslant 0$ on $T_{0} = \partial \mathcal{B}$, 
    completing our proof. 
   
    Observe that the expansion \eqref{asymp3} implies $\mathrm u_2$ is not identically zero. Thus, using 
    the strong maximum principle and the last equation in \eqref{second_order_system}, we see 
    that $\mathrm u_2> 0$ on $\mathbb{R}^n \backslash \Gamma$. Working backwards, the inequality $\mathrm u_2> 0$ and 
    the same reasoning implies $\mathrm u_1 > 0$ on $\mathbb{R}^n \backslash \Gamma$, which then in turn 
    gives us $\mathrm u_0>0$ on $\mathbb{R}^n \backslash \Gamma$. 
    
    The singular set $\Gamma$ is compact, so there exists $R_0> 0$ such that $\Gamma 
    \subset \mathbb{B}_{R_0} (0)$. We use the extended maximum principle \cite[Theorem~3.4]{MR0350027}
    to conclude there exists $\delta> 0$, depending on $R>R_0$, such that 
    \begin{equation} \label{apriori_bound1}
    \left. \mathrm u_0 \right |_{\mathbb{B}_R(0) \backslash \Gamma} \geqslant \delta, \quad 
    \left. \mathrm u_1 \right |_{\mathbb{B}_R(0) \backslash \Gamma} \geqslant \delta, \quad {\rm and} \quad
    \left. \mathrm u_2 \right |_{\mathbb{B}_R(0)\backslash \Gamma} \geqslant \delta. 
    \end{equation} 
  
    Combining our expansion \eqref{asymp3} with Lemma 2.3 of \cite{MR982351} there 
    exists $R_1>0$ and $\lambda_1\leqslant \lambda_0$ such that for each $\lambda < \lambda_1$ we have 
    $$\mathrm w_0^\lambda (x) > 0, \quad \mathrm w_1^\lambda(x) > 0, \quad {\rm and} \quad \mathrm w_2^\lambda (x) > 0 \quad
    \textrm{ for } \quad  x \in \Sigma_\lambda \quad \textrm{ and }\quad |x|> R. $$ 
    Using this inequality together with \eqref{apriori_bound1} then implies that there exists
    $\lambda_2 \leqslant \lambda_1$ such that 
    $$\mathrm w_0^\lambda(x) > 0, \quad \mathrm w_1^\lambda (x) > 0, \quad {\rm and} \quad \mathrm w_2^\lambda(x) > 0 \quad  
    \textrm{ on } \quad  \Sigma_\lambda^{\prime}  \quad
    {\rm for \ each } \quad \lambda \leqslant \lambda_2.$$   
    By construction 
    \begin{equation} \label{meanconvex4} 
    \Delta \mathrm w_2^\lambda (x) = c_n \left(\mathrm u_0(x^\lambda)^{\frac{n+6}{n-6}}
    - \mathrm u_0(x)^{\frac{n+6}{n-6}}\right) < 0 \quad \textrm{ on } \quad \Sigma_\lambda^{\prime} 
    \quad {\rm for \ each} \quad \lambda \leqslant \lambda_2.
    \end{equation} 
    On the other hand, the 
    asymptotic expansion \eqref{asymp3} implies 
    \begin{equation} \label{meanconvex5}
    \mathrm w_2^\lambda (x) \rightarrow 0 \quad \textrm{ as } \quad |x| \rightarrow \infty.
    \end{equation}
    Putting together \eqref{meanconvex4}, \eqref{meanconvex5} and 
    $\left. \mathrm w_2^\lambda \right |_{T_\lambda} = 0$, we see by the 
    maximum principle that $\mathrm w_2^\lambda (x) \geqslant 0$ for each $x \in 
    \Sigma_\lambda^{\prime}$ and $\lambda \leqslant \lambda_2$. 
    However, by the completeness of the metric $g$ on $\Omega$ we know that $\mathrm w_2^\lambda$ is not 
    identically zero on $\Sigma_\lambda^{\prime}$, so again the
    maximum principle actually implies $\mathrm w_2^\lambda (x) > 0$ 
    for each $x \in \Sigma_\lambda^{\prime}$ and $\lambda \leqslant 
    \lambda_2$. Once again, analogous arguments imply $\mathrm w_1^\lambda >0$ 
    and $\mathrm w_0^\lambda > 0$ on $\Sigma_\lambda^{\prime}$ for each 
    $\lambda \leqslant \lambda_2$. 
    
    At this point we define 
    $$\lambda^* = \sup \{ \lambda \leqslant 0 : \mathrm w_i^\mu(x) > 0
    \textrm{ for each } \mu \leqslant \lambda \textrm{ and }i=0,1,2 \} $$ 
    and prove that $\lambda^* = 0$. Following our definitions we have 
    $$\Delta \mathrm w_0^\lambda (x) = -\Delta \mathrm u_0(x) + \Delta \mathrm u_0(x^\lambda) <0$$ 
    for each $x \in \Sigma_\lambda^{\prime}$ and $\lambda < \lambda^*$, and 
    so $\Delta \mathrm w_0^{\lambda^*} \leqslant 0$ on $\Sigma_{\lambda^*}^{\prime}$. By 
    similar arguments we also have 
    $$\Delta \mathrm w_1^{\lambda^*} \leqslant 0, \quad \Delta \mathrm w_2^{\lambda^*} 
    \leqslant 0 \textrm{ on }\Sigma_{\lambda^*}^{\prime}.$$ 
    Now suppose $\lambda^* < 0$ and let $x^* \in \overline{\Sigma_{\lambda^*}^{\prime}}$ 
    such that $\mathrm w_i^{\lambda^*} (x^*) = 0$ for some $i=0,1,2$. If $x^{\lambda^*} 
    \in \Sigma_{\lambda^*}^{\prime}$ is an interior point then the maximum principle 
    implies $\mathrm w_i^{\lambda^*} \equiv 0$, which in turn means $\mathrm u_i$ is 
    symmetric about the hyperplane $T_{\lambda^*}$. This is impossible because the 
    singular set $\Gamma$ lies to one side of $T_{\lambda^*}$. On the other 
    hand, if $x^* \in T_{\lambda^*}$ then by the Hopf boundary lemma (together with the 
    fact that $\mathrm w_i^{\lambda^*}$ may not be constant in $\Sigma_{\lambda^*}^{\prime}$) 
    we have 
    \begin{equation}\label{boundary1} 
    0 < \partial_{x_1} \mathrm w_i^{\lambda^*} (x^*) = 2 \partial_{x_1}
    \mathrm u_i (x^*).\end{equation} 
    However, the asymptotic expansions \eqref{asymp1}, \eqref{asymp2} 
    and \eqref{asymp3} combined with 
    $\lambda^* < 0$ tells us 
    \begin{equation} \label{boundary2} 
    \mathrm u_i (x) - \mathrm u_i (x^{\lambda^*}) \geq \delta_3 \textrm{ for }
    |x| > R_2 \textrm { and } x_1 = \lambda^*
    \end{equation} 
    for some positive numbers $\delta_3$ and $R_2$. Combining \eqref{boundary1} 
    and \eqref{boundary2} implies the inequality $\mathrm w_i^\lambda$ 
    continues to hold for some small value $\lambda < \lambda^*$, contradicting 
    the definition of $\lambda^*$. 
	\end{proof}
	
	First, we prove the upper bound estimate.
	Our proof borrows from Pollack's proof of the corresponding upper bound in the scalar curvature case. 
	\begin{proposition}\label{prop:upperestimate}
		Let $u\in \mathcal{C}^{\infty}(\Omega)$ be a positive singular solution to \eqref{ourequation}.
		There exists $C_1>0$ depending only on $n$ and $d$ satisfying
		\begin{equation*}
		    u(x) \leqslant C_1\ud_{g_0}(x, \Lambda)^{-\gamma_n}.
		\end{equation*}
	\end{proposition}

	\begin{proof}
    Let $p_0 \not \in \Lambda$, and $\rho>0$ such that $\mathcal{B}_\rho(p_0)
    \subset \Omega$, where $\mathcal{B}_\rho(p_0)$ is a geodesic ball 
    with respect to the round metric. 
    We define the auxiliary function $\psi_{\rho}: \mathcal{B}_\rho(p_0) \rightarrow \mathbb R$ given by
    \begin{equation*}
        \psi_{\rho}(x) = (\rho-\ud_{g_0}(x, x_0))^{\gamma_n} u(x) .
    \end{equation*} 
    Notice that choosing $\rho = \frac{1}{2}\ud_{g_0} (x_0, \Lambda)$, it follows 
    \begin{equation*} 
    \psi_{\rho}(x_0) = \rho^{\gamma_n} u(x_0) = 2^{-\gamma_n}
    \ud_{g_0} (x_0, \Lambda)^{\gamma_n} u(x_0).
    \end{equation*} 
    
    We claim that there exists $C>0$ depending only on $n$ such that 
    $\psi_{\rho}(x) \leqslant C$ for all admissible choices of $\lambda$, $u$, 
    $x_0$, and $\rho$. 
    We suppose by contradiction that one can find sequences $\{\Lambda_k\}_{k\in\mathbb{N}}$, $\{u_k\}_{k\in\mathbb{N}}$, $\{p_{0,k}\}_{k\in\mathbb{N}}$, and $\{\rho_k\}_{k\in\mathbb{N}}$ of admissible parameters satisfying
    \begin{equation*}
        M_k = \psi_{\rho}(p_{1,k}) = \sup_{x \in \mathcal{B}_{\rho_k}(p_{0,k})} \psi_{\rho}(x) 
    \rightarrow +\infty. 
    \end {equation*} 
    Also, we observe $\left. \psi_{\rho} \right |_{\partial \mathcal{B}_{\rho_k} (p_{0,k})} = 0$, so $p_{1,k}\in {\rm int}(\mathcal{B}_{\rho_k} (p_{0,k}))$. 
    Next, by taking $r_k= \rho_k-\ud_{g_0}(p_{1,k}, p_{0,k})$, and defining be geodesic normal coordinates centered at $p_{1,k}$, denoted by $y$, we set 
    \begin{equation*}
     \lambda_k = 2 u_k(p_{1,k})^{-\gamma_n}, \quad
    R_k = {r_k}{\lambda_k^{-1}} = 2^{-1}{r_k}(u_k (p_{1,k}))^{-\gamma_n} 
    = 2^{-1} M_k^{1/\gamma_n}.
    \end{equation*} 
    We now construct a blow-up sequence $\{ w_k\}_{k \in \mathbb{N}} \subset 
    \mathcal{C}^{6,\alpha}(\mathbb{B}_{R_k})$ for 
    some $\alpha \in (0,1)$ by $ w_k: \mathbb{B}_{R_k}(0) \rightarrow \mathbb{R}$ is such that 
    \begin{equation*}
        w_k (y) = 
    \lambda_k^{\gamma_n} u_k(\lambda y) \quad {\rm for \ all} \quad k\in\mathbb N.
    \end{equation*} 
    Whence, using the conformal invariance in Remark \ref{rmk:scalinglaw}, one can verify that the 
    function $w_k\in \mathcal{C}^{6,\alpha}(\mathbb{B}_{R_k})$ satisfies
    \begin{equation*}
       P^6_{\lambda g_k}w_k = c_nw_k ^{\frac{n+6}{n-6}}  \quad {\rm in} \quad \mathbb{B}_{R_k}.
    \end{equation*}
    Moreover, by construction, one has
    \begin{equation*}
        2^{\gamma_n} = w_k(0) = \sup_{\mathbb{B}_{R_k}(0)} w_k(x) \quad {\rm for \ all} \quad k\in\mathbb N,
    \end{equation*}
    which, by Arzela--Ascoli theorem, means there exists subsequence that converges uniformly on 
    compacts. 
    
    In addition, it is not hard to check that the rescaled metrics $\lambda g_0$ 
    converge to the classical Euclidean metric $\delta$  as $k \rightarrow 
    \infty$.
    Therefore, by taking the limit of the blow-up sequence, we
    obtain a positive function $w_{\infty}\in \mathcal{C}^{6,\alpha}(\mathbb{R}^n)$ 
    satisfying $w_{\infty}(0) = \sup w_{\infty} = 2^{\gamma_n}$ and
    \begin{equation*} 
     (-\Delta)^3w_{\infty}= c_nw_{\infty}^{\frac{n+6}{n-6}} \quad {\rm in} \quad \mathbb{R}^n.
    \end{equation*}
    By the classification theorem in Theorem~\ref{thmA} (a), we must have 
    \begin{equation*}
        w_{\infty}(x) := {2}^{-\gamma_n}\left ({1+|x|^2}\right )^{-\gamma_n}={2}^{-\gamma_n}u_{\rm sph}(x).
    \end{equation*}
    Thus each solution $u_k$ has a bubble for $k\gg1$
    sufficiently large.
    In other terms, a small neighborhood of $p_{1,k}$ is close
    (in $\mathcal{C}^{6,\alpha}$-norm) to the round metric, and hence 
    has a concave boundary, for $k\gg1$ sufficiently large.
    
    We verify this by computing the 
    mean curvature of a geodesic sphere explicitly. Using 
    $g_0 = {4}{(1+|x|^2)^{-2}} \delta$, a direct computation shows 
    the mean curvature of a hypersurface is given by $ H_\Sigma = - \operatorname{tr}_g 
    \langle \nabla_{\partial \ell} \nu_\Sigma, \partial_m \rangle$, where $\nu_\Sigma$ is 
    the unit inward normal vector of $\Sigma$.
    
    A geodesic sphere centered at $p=0$ coincides with a Euclidean round sphere 
    centered at the origin (with a different radius), and so 
    \begin{equation*}
        \nu_{\Sigma} = - \left ( \frac{1+|x|^2}{2|x|} \right ) x^\ell \partial_{x_\ell}.
    \end{equation*}
    A straightforward computation yields
    \begin{equation*}
        H_{\Sigma} = -2n|x|(1+|x|^2)  + \frac{n-1 + n|x|^2}{|x|},
    \end{equation*}
    which is negative when $|x|>3$. 
    Additionally, since 
    \begin{equation*}
         \lim_{k\rightarrow\infty}\| w_k - w_{\infty}\|_{\mathcal{C}^{6,\alpha}(B_{{3R_k}/{4}}(0))}=0,
    \end{equation*}
    it holds that $\partial B_{{3R_k}/{4}}(0)$ is also mean concave with respect to the metric $\hat{g}_k\in{\rm Met}^{\infty}(B_{{3R_k}/{4}}(0))$ defined as $\hat{g}_k=w_k^{{4}/{(n-6)}}
    \delta_{\ell m}$, which in turn implies $\partial \mathbb{B}_{{3|p_{1,k}|}/{8}} (p_{1,k})$ is mean concave 
    with respect to the metric $g_k \in{\rm Met}^{\infty}(\Omega)$ given by $\hat{g}_k=u_k^{{4}/{(n-6)}} \delta_{\ell m}$.
    This is 
    contradiction with Lemma~\ref{lm:geodesicallyconvex}, which proves the claim.
	\end{proof}
	
    Second, we obtain a lower bound estimate.
	
	\begin{proposition}\label{prop:lowerestimate}
		Let $u\in \mathcal{C}^{\infty}(\Omega)$ be a positive singular solution to \eqref{ourequation}.
		There exists $C_2>0$ depending only on $u$ satisfying
		\begin{equation*}
		C_2\min_{j\in I_N}\ud_{g_0}(x, p_j)^{-\gamma_n}\leqslant u(x).
		\end{equation*}
	\end{proposition}
	
	\begin{proof}
	    Indeed, notice that by applying \cite[Theorem~1.3]{arXiv:2210.04619} in cylindrical coordinates $v=\mathfrak{F}(u)$, we obtain that $\mathcal{P}_{\rm cyl}(v) \leqslant 0$ with equality if and only if 
	    \begin{equation*}
	        \liminf_{t \rightarrow \infty} v(t,\theta) = \limsup_{t \rightarrow \infty} v(t,\theta) = \lim_{t \rightarrow \infty} v(t,\theta) = 0.
	    \end{equation*}
        Otherwise, if $\mathcal{P}_{\rm cyl}(v) < 0$, there exists $C_2> 0$, which depends on the solution $v$, such that $v(t,\theta) \geqslant C_2$.
        This proves the proposition.
    \end{proof}
 
	Third, we have a version of Harnack inequality for our setting, which will be important in the proof of our main result.
	
	\begin{proposition}\label{lm:harnackinquality}
	    Let $\Omega\subset\mathbb R^n$ and $u\in \mathcal{C}^{\infty}(\Omega)$.
	    Assume that $-\Delta u\geqslant 0$, $\Delta^2 u\geqslant 0$, and \begin{equation*}
	        (-\Delta)^3 u = f(u),
	    \end{equation*}
	    where $f$ is either linear or superlinear and $f(0) = 0$. 
	    Then, there exists $\rho_0>0$ such that for $\rho  \in(0,\rho_0]$ and $C_3>0$ depending only on $\Omega$, $f$, and $\rho$, it holds
        \begin{equation*}
            \sup_{\mathcal{B}_\rho(0)} u \leqslant C_3 \inf_{\mathcal{B}_\rho(0)} u.
        \end{equation*} 
	\end{proposition}

	\begin{proof}
		The proof is a straightforward adaptation of \cite[Theorem~3.6]{MR2240050}. 
	\end{proof}

	\section{Compactness result}\label{sec:mainresult}
	In this section, we prove the main result of the manuscript. 
	
	Before proceeding to the proof, we need to obtain the existence of a positive Green function for the sixth order GJMS of the round sphere with a prescribed asymptotic rate near a pole given by the fundamental solution to the flat tri-Laplacian.
    
    \begin{proposition}\label{prop:greenfunction}
    Let $p\in \Lambda\subset\mathbb (\mathbb{S}^n,g_0)$ be a point on the standard round sphere.
    There exists a Green function with pole at $p$, denoted by $G_p:\mathbb S^n\setminus \{ p \} \rightarrow (0,\infty)$, that 
    satisfies
    \begin{equation*}
        P_{g_0}G_p = \delta_p,
    \end{equation*}
     where $P_{g_0}$ is the sixth order GJMS operator of the round metric given by \eqref{sphericalGJMSoperator}  and $\delta_p$ is the Dirac function concentrated at $p$. 
    Furthermore, there exists $C_n>0$ depending only on $n$ such that
    \begin{equation}\label{green_estimate} 
        G_p(x)=C_n\ud_{g_0}(x,p)^{6-n} +\mathcal{O}(1) 
    \end{equation} 
    in conformal normal coordinates. 
    \end{proposition} 
    
    \begin{proof}
		This is a direct application of \cite[Proposition~2.1]{MR3652455} for the standard round sphere $(\mathbb S^n,g_0)$.
	\end{proof}

	\begin{proof}[Proof of Theorem~\ref{maintheorem}]
	Let $\{g_k\}_{k\in\mathbb N} = \{ (U_k)^{{4}/{n-6}} g_0 \} \subset \mathcal{M}_N^6$ 
	be a sequence of admissible metrics, each of which is a complete, conformally flat 
	metric on $\Omega_k = \mathbb{S}^n \backslash \Lambda_k$ with $Q^6_{g_k}\equiv Q_n = \frac{n(n^4-20n^2+64)}{32}$.
	We denote the punctures of $g_k$ by 
    \begin{equation*}
        \Lambda_k:={\rm sing}(U_k)=\{p_{1,k} , \dots, p_{N,k}\}\subset\mathbb S^n .
    \end{equation*}
	The proof will be divided into a sequence of steps.

    The first step will simplify our later analysis since it allows us to assume the singular points are fixed. 

	\noindent{\bf Step 1.} After passing to a subsequence, we may assume that for $k\gg1$ 
	sufficiently large each $U_k$ is non-singular on the set 
	$K_1:=\mathbb S^n \setminus (\cup_{i=1}^N \mathcal{B}_{\delta_1/2}(p_{j,i}))$.

    Indeed, for $0<\delta_1$ small enough, the set 
    \begin{equation*}
        (\mathbb S^n)^N \setminus \left \{(q_1, \dots, q_k) \in (\mathbb S^n)^N : 
    \ud_{g_0} (q_j, q_\ell) \geqslant \delta_1 \textrm{ for each }
    j \neq \ell \right \}
    \end{equation*}
    is compact and contains each singular set $\Lambda_k$ for all $k\in\mathbb N$. 
    Thus, there exits $\{p_{1,\infty},\dots,p_{N,\infty}\}\subset\mathbb S^n$, and a convergent subsequence such that
    $p_{j,k} \rightarrow p_{j,\infty}$ as $k\rightarrow+\infty$, proving Step 1.

    To set notation, we define the compact sets 
    \begin{equation*}
        K_\ell := \mathbb S^n \setminus \left ( \cup_{j=1}^N 
    \mathcal{B}_{2^{-\ell} \delta_1} (p_{j,\infty}) \right ) \quad {\rm for \ each} \quad  \ell \in \mathbb{N}
    \end{equation*}  
    Notice that by construction the family 
    $\{ K_\ell\}_{\ell\in\mathbb N}$ is a compact exhaustion of the limit singular set
    \begin{equation*}
        \Omega_{\infty}:=\mathbb S^n \setminus \Lambda_{\infty}, \quad {\rm where} \quad \Lambda_{\infty}:=\{ p_{1,\infty}. 
    \dots, p_{k,\infty}\}.
    \end{equation*}
    Furthermore, by the convergence $p_{j,k} \rightarrow p_{j,\infty}$ as $k\rightarrow+\infty$, 
    for each fixed $\ell\in\mathbb N$ there exists 
    $k_0\gg1$ such that $k \geqslant k_0$ implies $U_k$ is smooth in $K_\ell$.
    
    The second step is based on the uniform upper bound and states that we can extract a limit.
    
    \noindent{\bf Step 2.} The exists $U_{\infty}\in \mathcal{C}^{\infty}(\Omega_{\infty})$ solving \eqref{ourequation} such that 
    \begin{equation}\label{convergence}
        \lim_{k\rightarrow+\infty}\|U_{\infty}-U_k\|_{\mathcal{C}_{\rm loc}^{\infty}(\Omega_{\infty})}=0.
    \end{equation}
    
    \noindent In fact, using the upper bound in Proposition~\ref{prop:upperestimate}, one has that for each compact subset $K \subset 
    \Omega_{\infty}$, there exists $\alpha \in (0,1)$ and $C_1>0$
    depending only on $n$, $\Omega$, and $\alpha$ such that 
    \begin{equation*}
        \|U_k\|_{\mathcal{C}{6,\alpha} (K)} \leqslant C_1 \quad {\rm for \ all} \quad k\in\mathbb N.
    \end{equation*}
    Therefore, as a consequence of the Arzela--Ascoli theorem,
    one can find a limit $U_{\infty}\in \mathcal{C}^{6,\alpha}(K)$ a convergent subsequence, which we again denote the same, such that 
    \begin{equation*}
        \lim_{k\rightarrow+\infty}\|U_{\infty}-U_k\|_{\mathcal{C}_{\rm loc}^{6,\alpha}(\Omega_{\infty})}=0.
    \end{equation*}
    Furthermore, by applying standard elliptic regularity, we directly obtain that \eqref{convergence} holds, and so Step 2 is proved.
    
    The next step is to show that this limit is non-trivial.

	\noindent{\bf Step 3.} $U_{\infty}>0$ on $\Omega_{\infty}$.

    \noindent If this step were false, there would exist 
    $p_* \in \Omega_{\infty}$ such that 
    \begin{equation*}
        0 = U_{\infty} (p_*) = \lim_{k \rightarrow +\infty} U_k(p_*).
    \end{equation*}
    For each $k\in\mathbb N$, we define $\varepsilon_k = U_k(p_*)$ and the rescaled function 
    $\widehat{U}_k\in \mathcal{C}^{\infty}(\Omega_k)$ given by 
    \begin{equation*}
        \widehat{U}_k (x) = {\varepsilon_k}^{-1}U_k(x)  \quad {\rm for \ all} \quad k\in\mathbb N.
    \end{equation*}
    As a consequence of Remark~\ref{rmk:scalinglaw}, it follows 
    \begin{equation*}
        P_{g_0}\widehat{U}_k = \varepsilon_k^{\frac{12}{n-6}} c_n \widehat{U}_k^{\frac{n+6}{n-6}} 
        \quad {\rm in} \quad \Omega_k \quad {\rm for \ all} \quad k\in\mathbb N.
    \end{equation*} 
    In addition, by construction, the sequence $\{\widehat{U}_k\}_{k\in\mathbb N}$ satisfy the normalization 
    \begin{equation} \label{normalization} 
        \widehat U_k(p_*) = 1 \quad {\rm for \ all} \quad k\in\mathbb N.
    \end{equation} 
    
    By the Harnack inequality of Lemma~\ref{lm:harnackinquality} there exists a 
    positive constant $C_1$ depending only on $n$ and $\ell$ such that 
    \begin{equation} \label{final_step1}
    \sup_{K_\ell} |u_{\rm sph} \widehat U_k| \leqslant C_1.
    \end{equation} 
    However, there is another positive constant $C_2$, again depending only 
    on $n$ and $\ell$, such that 
    \begin{equation} \label{final_step2}
    C_2 \leqslant u_{\rm sph} \leqslant 2^{\gamma_n}. 
    \end{equation} 
    Combining \eqref{final_step1} and \eqref{final_step2} there 
    exists a uniform constant $C_3$ such that 
    $$\sup_{K_\ell} \widehat U_k \leq C_3,$$ 
    and so by the Arzela-Ascoli theorem we may pass to a 
    subsequence $\widehat U_k$ that converges uniformly on compact 
    subsets of $\Omega_\infty$ to a smooth function $\widehat U_\infty$.
    
    This limit function $\widehat U_\infty:\Omega_\infty \rightarrow \mathbb{R}$ satisfies 
    \begin{equation*}
        P_{g_0} 
    \widehat U_\infty = 0 \quad {\rm in} \quad \Omega_\infty
    \end{equation*}
    and so it has the form 
    $$\widehat U_\infty = \sum_{j=1}^N \beta_j G_{p_{j, \infty}}$$ 
    for some collection of real numbers $\beta_1, \dots, \beta_N$. 
    The normalization \eqref{normalization} implies one of the 
    coefficients $\beta_{j_0}$ is positive, so after possibly relabeling 
    the punctures we may assume $\beta_1 > 0$. 
    
    We now choose a stereographic projection sending $p_{1,\infty}$ 
    to the origin and perform the Emden-Fowler change of coordinates in 
    Definition~\ref{def:cylindricaltransformation}, which yield the 
    functions 
    \begin{equation*}
        v_k:=\mathfrak{F}(u_{\rm sph} U_k) \quad {\rm and} \quad 
        \widehat{v}_k:=\mathfrak{F}(u_{\rm sph} \widehat{U}_k)
    \end{equation*}
    and their respective limits 
    \begin{equation*}
        v_\infty:=\mathfrak{F}(u_{\rm sph} U_\infty) \quad {\rm and} 
        \quad \widehat{v}_\infty:=\mathfrak{F}(u_{\rm sph}\widehat{U}_\infty).
    \end{equation*}
    
    The expansion \eqref{green_estimate} implies 
    \begin{eqnarray} \label{final_step3}
        \widehat{v}_{\infty} (t,\theta)=  e^{-\gamma_nt}(\cosh t)^{-\gamma_n}(C_ne^{-\gamma_nt} 
        + \mathcal{O}(1))=C_n + \mathcal{O}(e^{(6-n)t}) \quad {\rm as} \quad t\rightarrow+\infty. 
    \end{eqnarray} 
    Also, observe that $\widehat{v}_k\in \mathcal{C}^{6}(\mathcal{C}_T)$ 
    satisfies the PDE 
    \begin{equation*}
        P^6_{\rm cyl} \widehat{v}_k = \varepsilon_k^{\frac{12}{n-6}}c_n\widehat{v}_k^{\frac{n+6}{n-6}} 
        \quad {\rm in} \quad \mathcal{C}_{T_k} \quad {\rm for \ all} \quad k\in\mathbb N,
    \end{equation*}
    which we combine with \eqref{final_step3} and Proposition~\ref{prop:rescaled_necksize_poho} and the convergence 
    $\widehat v_k \rightarrow \widehat v_\infty$ to see that for $t$ sufficiently large 
    \begin{eqnarray}\label{final_step4}
    \int_{\{ t \} \times \mathbb{S}^{n-1}} \mathcal{H}_{\rm cyl}^{\varepsilon_k^{\frac{12}{n-6}}c_n}
    \ud\theta & = & \int_{\{ t \} \times \mathbb{S}^{n-1}} \mathcal{H}_{\rm rad}(\widehat v_k)+ 
    \mathcal{H}_{\rm ang} (\widehat v_k) + \frac{n-6}{2n} \varepsilon^{\frac{12}{n-6}} c_n |\widehat v_k|^{\frac{2n}{n-6}}
    \ud\theta  \\ \nonumber  
    & = &  -\widetilde C_n \beta_1^2 + \mathcal{O}(e^{(6-n)t}) . \end{eqnarray} 
    for some $\tilde{C}_n>0$.  
    On the other hand, by our construction we have
    \begin{eqnarray} \label{final_step5} 
     \mathcal{P}_{\rm cyl}(v_k) & = & \int_{\{ t \} \times \mathbb{S}^{n-1}} 
     \mathcal{H}_{\rm rad} (v_k) + \mathcal{H}_{\rm ang} (v_k) + F(v_k) \ud \theta \\ \nonumber 
     & = & \int_{\{ t \} \times \mathbb{S}^{n-1}} \varepsilon_k^2 \left ( \mathcal{H}_{\rm rad} 
     (\widehat v_k) + \mathcal{H}_{\rm ang} (\widehat v_k)\right ) + \varepsilon_k^{\frac{2n}{n-6}}
     F(\widehat v_k) \ud \theta \rightarrow 0 
     \end{eqnarray}
    
    From \eqref{final_step4} and \eqref{final_step5}, we find
    \begin{equation*}
        \lim_{k \rightarrow +\infty} \mathcal{P}_{\rm rad}(g_k, p_{1,k})= 0,
    \end{equation*}
    which, together with Proposition~\ref{prop:necksizexpohozaev}, implies $  \lim_{k \rightarrow +\infty} \varepsilon_1(g_k) = 0.$
    This contradicts the hypothesis that the necksizes are bounded away from zero, that is, $\varepsilon_j(g_k)>\delta_1$ for some $0<\delta_1\ll1$. 

    At last, we can complete our argument
	
	\noindent{\bf Step 4.} The metric $g_{\infty}= U_\infty^{\frac{4}{n-6}} g_0$ is a 
	complete metric on $\Omega_{\infty}$.
	
    \noindent Indeed, suppose by contradiction that is $g_{\infty}$ is incomplete.
    Then there exists an index $j \in \{ 1,\dots, N\}$ such that 
    $\liminf_{x \rightarrow p_{j,\infty}} {U}_{\infty}(x) < \infty$.
    In this case, the removable singularity result in Proposition~\ref{prop:lowerestimate} implies 
    \begin{equation*}
        \mathcal{P}_{\rm rad} 
    (g_{\infty}, p_{j,\infty}) = 0.
    \end{equation*}
    However, by construction
    \begin{equation*}
        0=\mathcal{P}_{\textrm{rad}}(g_{\infty}, p_{j,\infty}) = \lim_{k \rightarrow +\infty} 
    \mathcal{P}_{\rm rad} (g_k, p_{j,k}) \geqslant \delta_2,
    \end{equation*}
    which, by Proposition~\ref{prop:necksizexpohozaev} implies $\varepsilon_j(g_k)\geqslant\delta_2$, which is contradiction with the fact $g_k\in \mathcal{Q}^6_{\delta_1,\delta_2}$.
    
    By putting all these steps together, our main theorem is proved.
	\end{proof}
	
	\appendix
	
	\section{Higher order curvature tensors}
	    Let $(M^{n},g)$ is a Riemannian manifold with $n\geqslant 2$.
	    In what follows, we will always be using Einstein's summation convection.    
		In a local coordinate frame, denoted by $\{\partial_i\}_{i=1}^n$, we let ${\rm Rm}_g\in\mathfrak{T}^3_1(M)$ be the {Riemannian curvature tensor},  $\accentset{\circ}{\rm Rm}_g\in\mathfrak{T}^4_0(M)$ be {covariant Riemann curvature tensor}, and the {Ricci curvature tensor} ${\rm Ric}_g={\rm tr}_g\accentset{\circ}{\rm Rm}_g\in\mathfrak{T}^2_0(M)$, which can be expressed as $	{\rm Ric}_{jk}=\accentset{\circ}{\rm Rm}_{i jk}^{i}=g^{i\ell} \accentset{\circ}{\rm Rm}_{ijk\ell}$.
		We also consider the {scalar curvature} $R_g={\rm tr}_g{\rm Ric}_g\in\mathfrak{T}^0_0(M)$, defined by
		$R=g^{ij}{\rm Ric}_{ij},$ where $\mathfrak{T}^r_s(M)$ stands for the set of $(r,s)$-type tensor over $M$  with $\mathfrak{T}^0_0(M)=\mathcal{C}^{\infty}(M)$ and ${\rm tr}_g:\mathfrak{T}^r_s(M)\rightarrow\mathfrak{T}^{r-2}_s(M)$.
		Also for  the Laplace--Beltrami operator, we simply denote $\Delta_{g}:=g^{i j} \nabla_{i} \nabla_{j}$, where $\nabla_g$ the Levi--Civita connection associated to $g$.
		
		It is also convenient to define some operations involving two tensors. 

        \begin{definition}
            First, let us introduce the {cross product} $\times:{\rm Sym}_{2}(M)\times{\rm Sym}_{2}(M)\rightarrow {\rm Sym}_{2}(M)$ is given by
		\begin{equation*}
			(h_1 \times h_2)_{ij}:=g^{k\ell} h_{1,ik} h_{2,j\ell}=h_{1,i}^{\ell} h_{2,\ell j}.
		\end{equation*}
		
		Second, we define a {\it dot product} $\times:{\rm Sym}_{2}(M)\times{\rm Sym}_{2}(M)\rightarrow \mathbb{R}$, given by
		\begin{equation*}
			h_1 \cdot h_2:={\rm tr}_g(h_1 \times h_2)=g^{ij} g^{k\ell} h_1^{ik} h_{2,j\ell}=h_1^{jk} h_{2,jk}.
		\end{equation*}
		
		Third, we also recall the {\it Kulkarni--Nomizu product} $\owedge:{\rm Sym}_{2}(M)\times{\rm Sym}_{2}(M)\rightarrow \mathfrak{T}^4_0(M)$ 
			\begin{equation*}
			(h_1\owedge h_2)_{ijk\ell}:=h_{1,i\ell}h_{2,jk}+h_{1,jk}h_{2,i\ell}-h_{1,ik} h_{2,j\ell}-h_{1,j \ell}h_{2,ik}.
		\end{equation*}
		
    		 At last, we consider $\cdot:{\rm Sym}_{2}(M)\rightarrow{\rm Sym}_{2}(M)$ and $\delta_g:{\rm Sym}_{2}(M)\rightarrow\mathbb{R}$, 
    		\begin{equation*}
    			(\accentset{\circ}{\rm Rm} \cdot h)_{j k}:=R_{ijk\ell} h^{i\ell} \quad \mbox{and} \quad \left(\delta_{g} h\right)_{i}:=-\left(\operatorname{div}_{g} h\right)_{i}=-\nabla^{j} h_{ij},
    		\end{equation*}	
    		where the latter one is the $L^{2}$-formal adjoint of Lie derivative (up to scalar multiple).
		
        \end{definition}
		
		\begin{definition}\label{def:geometrictensors}
		    Let us define the {\it Schouten tensor}, {\it Weyl tensor}, {\it Bach tensor}, and {\it nameless tensor}, respectively, by
	\begin{align*}
		A_g&:=\frac{1}{n-2}\left(\operatorname{Ric}_{g}-\frac{1}{2(n-1)} R_{g} g\right)\\
		W_g&:=\accentset{\circ}{\rm Rm}_g-A_g\owedge g\\
		B_g&:=\Delta_{g} A_g-\nabla_g^{2} \operatorname{tr}_gA_g+2 \accentset{\circ}{\rm Rm}_g\cdot A_g-(n-4) A_g\times A_g-|A_g|^{2} g-2(\operatorname{tr}_g A_g) A_g,A	\end{align*}	
	where these expressions are written in an abstract index-free manner.
		\end{definition}
	From this, we introduce the following tensors
	   \begin{align*}
        T^2_{g}&:=(n-2)\sigma_1(A_g) g-8A_g,\\
        T^4_{g}&:=-\frac{3 n^2-12 n-4}{4} \sigma_1(A_g)^2 g+4(n-4)|A|_g^2 g+8(n-2) \sigma_1(A_g) A_g\\
        &+(n-6)\Delta_g \sigma_1(A_g) g+48A_g^2-\frac{16}{n-4} B_g,\\
        T^6_g&:=-\frac{1}{8}\sigma_3(A_g)-\frac{1}{24(n-4)}\langle B_g, A_g\rangle_g,
    \end{align*}
    where $\sigma_k$ is the $k$-th elementary symmetric function for each $k\in\mathbb N$.
    
    Based on this notation, we introduce the concept of higher order curvatures as follows
    \begin{definition}\label{def:curvatures}
        For any $g\in {\rm Met}^{\infty}(\Omega)$, let us define the $N$th order $Q$-curvature $Q_g^N$ for $N=2,4,6$, respectively, by
    \begin{align*}
        Q^2_g&:=R_g\\
        Q^4_g&:=-\frac{1}{2(n-1)}\Delta R_g-\frac{2}{(n-2)^2}|\Ric_g|^2
		+\frac{n^3-4n^2+16n-16}{8(n-1)^2(n-2)^2}R_g^2,&\\
        Q_g^6&:=-3 ! 2^6 T^6_g-\frac{n+2}{2} \Delta_g(\sigma_1\left(A_g\right)^2)+4 \Delta_g|A|_g^2-8 \delta\left(A_g \ud \sigma_1\left(A_g\right)\right)+\Delta_g^2 \sigma_1\left(A_g\right)\\
        &-\frac{n-6}{2} \sigma_1\left(A_g\right) \Delta_g \sigma_1\left(A_g\right) -4(n-6) \sigma_1\left(A_g\right)|A|_g^2+\frac{(n-6)(n+6)}{4} \sigma_1\left(A_g\right)^3.
    \end{align*}
	\end{definition}
    
    Associated with these curvatures, we have the following conformally invariant operators
    \begin{definition}\label{def:conformaloperator}
        For any $g\in {\rm Met}^{\infty}(\Omega)$, let us define the $N$th order GJMS operator $P_g^N$ for $N=2,4,6$, respectively, by
    \begin{align*}
        &P^2_g:=-\Delta_g+\frac{n-2}{2}R_g&\\
        &P^4_g:=\Delta_{g}^{2}-\dive\left(\frac{(n-2)^2+4}{2(n-1)(n-2)}R_{g}g-\frac{4}{n-2}\Ric_g\right)\ud +\frac{n-4}{2}Q^4_g&\\
        &P_g^6:=-\Delta_g^3-\Delta_g \delta T_2 \ud-\delta T_2 \ud \Delta_g-\frac{n-2}{2} \Delta_g\left(\sigma_1\left(A_g\right) \Delta_g\right)-\delta T_4 \ud+\frac{n-6}{2} Q_g^6.&
    \end{align*}
        When $N=2$, the operator $P_g^2=L_g$ is the so-called conformal Laplacian.
    \end{definition}

    \section{Modica estimates}
        In this appendix, we discuss possible pointwise estimates for positive smooth solutions to \eqref{ourlimitPDE}.
        These estimates have strong geometric implications in terms of the associated conformally flat metric.
        
        In \cite[Theorem~1.4]{arXiv:1708.04660}, it is proved that positive smooth solutions to 
        \begin{equation*}
            \Delta^2 u=\frac{n(n-4)(n^2-4)}{16}u^{\frac{n+4}{n-4}} \quad {\rm in} \quad \mathbb{R}^n\setminus\{0\}.
        \end{equation*}      
        satisfies the following pointwise inequality
        \begin{equation*}
            -\Delta u-\frac{4}{n-2}\frac{|\nabla u|^2}{u}\geqslant \sqrt{\frac{n-4}{n}}u^{\frac{n}{n-4}}  \quad {\rm in} \quad \mathbb{R}^n\setminus\{0\}.
        \end{equation*}
        This implies in particular that the scalar curvature $Q_g^2$ of the conformally flat metric $g=u^{4/(n-4)}\delta$ is positive.
        This type of result is known in the literature as Modica-type estimates.

       In our situation, we start by writing the metric $g\in [g_0]$ as $g =( u^{\frac{n-2}{n-6}})^{\frac{4}{n-2}} \delta$, we see 
    	\begin{equation} \label{scal_curv}
        	Q_g^2= -\frac{4(n-1)}{n-2} u^{\frac{-(n+2)}{n-6}} \Delta \left ( u^{\frac{n-2}{n-6}} \right ) 
        	= -\frac{4(n-1)}{n-6} u^{-\frac{n-2}{n-6}} \left ( \Delta u + \frac{4}{n-6}
        	\frac{|\nabla u|^2}{u} \right ) .
        \end{equation} 
        and
        \begin{equation*}
        	-\Delta \left ( u^{\frac{n-2}{n-6}} \right ) 
        	= -\Delta u -\frac{4}{n-6}
        	\frac{|\nabla u|^2}{u}.
        \end{equation*} 
        From this, we conclude that $Q_g^2 \geqslant 0$ implies $-\Delta u \geqslant  0$, and in fact is a stronger condition. 
        Similarly, writing $g=(u^{\frac{n-4}{n-6}})^{\frac{4}{n-4}} \delta$, it follows
    	\begin{equation} \label{q4-curv} 
        	Q_g^4 = \frac{2}{n-4} u^{-\frac{n+4}{n-6}} \Delta^2 \left ( u^{\frac{n-4}{n-6}} \right ).  
    	\end{equation} 
        Furthermore, a long computation shows 
    	\begin{eqnarray*} 
        	(-\Delta)^2 \left ( u^{\frac{n-4}{n-6}} \right ) & = & \frac{n-4}{n-6} u^{\frac{2}{n-6}} \Delta^2 u + 
        	\frac{8(n-4)}{(n-6)^2} u^{\frac{8-n}{n-6}} \langle \nabla u, \nabla \Delta u \rangle \\ 
        	& & + \frac{4(n-4)}{(n-6)^2} u^{\frac{8-n}{n-6}} |D^2 u|^2 + \frac{8(n-4)(8-n)}{(n-6)^3} 
        	u^{\frac{-2(n-7)}{n-6}} D^2u (\nabla u, \nabla u) \\ 
        	& & + \frac{4(n-4)(8-n)}{(n-6)^3} u^{\frac{-2(n-7)}{n-6}} |\nabla u|^2 \Delta u 
        	+ \frac{2(n-7)(n-8)}{(n-6)^4} u^{\frac{20-3n}{n-6}} |\nabla u|^4,
    	\end{eqnarray*} 
    	where $$|D^2 u|^2 = \sum_{i,j = 1}^n u_{x_ix_j}^2 \quad {\rm and} \quad D^2u (\nabla u, \nabla u) 
    	= \sum_{i,j= 1}^n u_{x_ix_j} u_{x_i} u_{x_j}.$$
        Hence, the conditions $Q_g^2 \geqslant 0$ and $Q_g^4 \geqslant 0$ are not enough to guarantee that $\Delta^2 u \geqslant  0$ directly.
   
        Based on this, it is natural to ask whether the following result holds.      
        \begin{conjecture}
         Let $u\in \mathcal{C}^{\infty}(\mathbb R^n\setminus\{0\})$ be a positive solutions to \eqref{ourlimitPDE}.
        Then, the conformally flat metric given by $g = u^{{4}/{(n-6)}} \delta$ satisfies the following pointwise estimate
        \begin{equation*}
            Q_2(u)\geqslant \sqrt{\frac{n-6}{n}}u^{\frac{n}{n-6}} \quad {\rm and}  \quad  Q_4(u)\geqslant \sqrt{\frac{n-6}{n}}u^{\frac{n}{n-6}} \quad {\rm in} \quad \mathbb{R}^n\setminus\{0\}
        \end{equation*}
        where
        \begin{equation*}
            Q_2(u):=-\Delta u-\frac{4}{n-6}\frac{|\nabla u|^2}{u}.
        \end{equation*}
        and
         \begin{align*}
            Q_4(u)&:=\Delta^2 u-\frac{8}{(n-6)} u^{\frac{8-n}{2}} \langle \nabla u, \nabla \Delta u \rangle- \frac{4}{(n-6)} u^{\frac{8-n}{2}} |D^2 u|^2 -\frac{8(8-n)}{(n-6)^2}u^{{7-n}} D^2u (\nabla u, \nabla u)\\ 
        	&- \frac{4(8-n)}{(n-6)^2} u^{{7-n}} |\nabla u|^2 \Delta u - \frac{2(n-7)(n-8)}{(n-6)^3(n-4)} u^{\frac{20-3n}{2}} |\nabla u|^4.
        \end{align*}
        In particular, it follows that the curvatures $Q^2_g$ and $Q^4_g$ associated with the conformally flat metric $g=u^{4/(n-6)}\delta$ are both positive.
        \end{conjecture}
 

\begin{thebibliography}{10}
    
    \bibitem{arXiv:2110.05234}
    J.~H. Andrade, R.~Caju, J.~M. do~\'O, J.~Ratzkin and A.~Silva~Santos, Constant
      $Q$-curvature metrics with Delaunay ends: the nondegenerate case, \emph{to
      appear in Ann. Scuola Norm. Sup. Pisa Cl. Sci.}  (2023).
    
    \bibitem{10.1093/imrn/rnab306}
    J.~H. Andrade, J.~M. do~Ó and J.~Ratzkin, Compactness within the space of
      complete, constant Q-curvature metrics on the sphere with isolated
      singularities, \emph{Int. Math. Res. Not. IMRN} {\bf 2022} (2021)
      17282--17302.
    
    \bibitem{arXiv:2210.04376}
    J.~H. Andrade and J.~Wei, Classification for positive singular solutions to
      critical sixth order equations, \emph{arXiv:2210.04376 [math.AP]}  (2022).
    
    \bibitem{MR4420104}
    W.~Ao, M.~d.~M. Gonz\'{a}lez, A.~Hyder and J.~Wei, Removability of
      singularities and superharmonicity for some fractional {L}aplacian equations,
      \emph{Indiana Univ. Math. J.} {\bf 71} (2022) 735--766.
    
    \bibitem{MR982351}
    L.~A. Caffarelli, B.~Gidas and J.~Spruck, Asymptotic symmetry and local
      behavior of semilinear elliptic equations with critical {S}obolev growth,
      \emph{Comm. Pure Appl. Math.} {\bf 42} (1989) 271--297.
    
    \bibitem{MR2240050}
    G.~Caristi and E.~Mitidieri, Harnack inequality and applications to solutions
      of biharmonic equations, \emph{Partial differential equations and functional
      analysis}, \emph{Oper. Theory Adv. Appl.}, vol. 168, Birkh\"{a}user, Basel
      (2006) 1--26.
    
    \bibitem{MR4285731}
    J.~S. Case and W.~Luo, Boundary operators associated with the sixth-order
      {GJMS} operator, \emph{Int. Math. Res. Not. IMRN} {\bf 21} (2021)
      10600--10653.
      
    \bibitem{chang-gonzalez} S.~Y. Chang and M.~M. Gonz\'alez, Fractional 
    Laplacian in conformal geometry, \emph{Adv. Math.} {\bf 226} (2011) 1410--1432.
    
    \bibitem{chang-han-yang}
    S.-Y. Chang, Z.-C. Han and P.~Yang, Some remarks of the geometry of class of
      locally conformally flat metrics, \emph{Progress in Mathematics} {\bf 333}
      (2020) 37--56.
    
    \bibitem{MR3652455}
    X.~Chen and F.~Hou, Remarks on {GJMS} operator of order six, \emph{Pacific J.
      Math.} {\bf 289} (2017) 35--70.

    \bibitem{MR3694655}
A.~DelaTorre, M.~del Pino, M.~d.~M. Gonz\'{a}lez and J.~Wei, Delaunay-type
  singular solutions for the fractional {Y}amabe problem, \emph{Math. Ann.}
  {\bf 369} (2017) 597--626.

    \bibitem{MR3073887}
    C.~Fefferman and C.~R. Graham, Juhl's formulae for {GJMS} operators and
      {$Q$}-curvatures, \emph{J. Amer. Math. Soc.} {\bf 26} (2013) 1191--1207.
    
    \bibitem{MR3073449}
    A.~R. Gover and B.~Oersted, Universal principles for {K}azdan-{W}arner and
      {P}ohozaev-{S}choen type identities, \emph{Commun. Contemp. Math.} {\bf 15}
      (2013) 1350002, 27.
      
    \bibitem{GJMS} C.~R. Graham, R. Jenne, D. Mason and G. Sparling, Conformally invariant 
    powers of the Laplacian I: existence, \emph{J. London Math. Soc.} {\bf 46} (1992) 557--565.
    
    \bibitem{graham-zworski} C.~R. Graham and M. Zworski, Scattering matrix in 
    conformal geomery, \emph{Invnt. Math.} {\bf 152} (2003) 89--118.
    
    \bibitem{arXiv:1708.04660}
    Z.~Guo, X.~Huang, L.~Wang and J.~Wei, On Delaunay solutions of a biharmonic
      elliptic equation with critical exponent, \emph{J. Anal. Math.} {\bf 140}
      (2020) 371--394.
    
    \bibitem{arXiv:2210.04619}
    X.~Huang, Y.~Li and H.~Yang, Super polyharmonic property and asymptotic
      behavior of solutions to the higher order Hardy-Hénon equation near isolated
      singularities, \emph{arXiv:2210.04619 [math.AP]}  (2022).
    
    \bibitem{arxiv:1901.01678}
    T.~Jin and J.~Xiong, Asymptotic symmetry and local behavior of solutions of
      higher order conformally invariant equations with isolated singularities,
      \emph{Ann. Inst. H. Poincar\'{e} Anal. Non Lin\'{e}aire}  (2020).
    
    \bibitem{MR3077914}
    A.~Juhl, Explicit formulas for {GJMS}-operators and {$Q$}-curvatures,
      \emph{Geom. Funct. Anal.} {\bf 23} (2013) 1278--1370.
    
    \bibitem{MR0350027}
    N.~S. Landkof, \emph{Foundations of modern potential theory}, Die Grundlehren
      der mathematischen Wissenschaften, Band 180, Springer-Verlag, New
      York-Heidelberg (1972), translated from the Russian by A. P. Doohovskoy.
    
    \bibitem{MR1425579}
    R.~Mazzeo and F.~Pacard, A construction of singular solutions for a semilinear
      elliptic equation using asymptotic analysis, \emph{J. Differential Geom.}
      {\bf 44} (1996) 331--370.
    
    \bibitem{MR1356375}
    R.~Mazzeo, D.~Pollack and K.~K. Uhlenbeck, Moduli spaces of singular {Y}amabe
      metrics, \emph{J. Amer. Math. Soc.} {\bf 9} (1996) 303--344.
    
    \bibitem{MR4438901}
    Q.~A. Ng\^{o} and D.~Ye, Existence and non-existence results for the higher
      order {H}ardy-{H}\'{e}non equations revisited, \emph{J. Math. Pures Appl.
      (9)} {\bf 163} (2022) 265--298.
    
    \bibitem{MR0192184}
    S.~I. Pohozhaev, On the eigenfunctions of the equation {$\Delta u+\lambda
      f(u)=0$}, \emph{Dokl. Akad. Nauk SSSR} {\bf 165} (1965) 36--39.
    
    \bibitem{MR1266101}
    D.~Pollack, Compactness results for complete metrics of constant positive
      scalar curvature on subdomains of {$S^n$}, \emph{Indiana Univ. Math. J.} {\bf
      42} (1993) 1441--1456.
    
    \bibitem{schoen-convexity} R.~Schoen, On the number of constant scalar curvature metrics 
    in a conformal class, \emph{Differential Geometry, Pitman Monogr. Surveys 
    Pure Appl. Math.} {\bf 52} (1991) 311--320.
    
    \bibitem{MR929283}
    R.~Schoen, The existence of weak solutions with prescribed singular behavior
      for a conformally invariant scalar equation, \emph{Comm. Pure Appl. Math.}
      {\bf 41} (1988) 317--392.
    
    \bibitem{MR931204}
    R.~Schoen and S.-T. Yau, Conformally flat manifolds, {K}leinian groups and
      scalar curvature, \emph{Invent. Math.} {\bf 92} (1988) 47--71.
    
    \bibitem{MR1679783}
    J.~Wei and X.~Xu, Classification of solutions of higher order conformally
      invariant equations, \emph{Math. Ann.} {\bf 313} (1999) 207--228.
    
    \end{thebibliography}

\end{document}